\documentclass[11pt]{amsart}
\pdfoutput=1

\usepackage{fullpage,url}
\usepackage{caption}
\usepackage{subcaption}

\usepackage[labelformat=simple]{subcaption}

\usepackage{lipsum}
\usepackage{amsfonts}
\usepackage{graphicx}
\usepackage{epstopdf}
\usepackage{algorithmic}
\ifpdf
  \DeclareGraphicsExtensions{.eps,.pdf,.png,.jpg}
\else
  \DeclareGraphicsExtensions{.eps}
\fi

\usepackage{enumitem}
\setlist[enumerate]{leftmargin=.5in}
\setlist[itemize]{leftmargin=.5in}

\usepackage{amsopn}

\numberwithin{equation}{section}

\newtheorem{theorem}{Theorem}
\numberwithin{theorem}{section}
\newtheorem{proposition}[theorem]{Proposition}
\newtheorem{lemma}[theorem]{Lemma}
\newtheorem{remark}[theorem]{Remark}
\theoremstyle{definition}
\newtheorem{definition}[theorem]{Definition}

\newtheorem{example}[theorem]{Example}
\newtheorem{assumption}[theorem]{Assumption}

\usepackage{multicol}

\usepackage{amsmath}
\usepackage{amsfonts}
\usepackage{amssymb}

\usepackage{stmaryrd}
\usepackage{graphicx}
\usepackage{cite}
\usepackage{algorithmic}
\usepackage{algorithm}
\usepackage{nicefrac,color}
\usepackage{exscale,relsize}

\newcommand{\tg}{\tilde{g}}

\newcommand{\rr}{\mathbb{R}}

\newcommand{\Fc}{\mathcal{F}}

\newcommand{\Oc}{\mathcal{O}}

\newcommand{\supp}{\operatorname{supp}}

\newcommand{\WF}{\mathrm{WF}}                         

\newcommand{\WFL}{\mathrm{WF}_{L(\theta,p)}}
\newcommand{\WFLo}{\mathrm{WF}_{L(\theta_0,p_0)}}

\newcommand{\WFpto}{\mathrm{WF}_{(\theta_0,p_0)}}

\newcommand{\NL}{N(L\thp)}
\newcommand{\NLo}{N(L\thpo)}

\newcommand{\RA}{R_A}
\newcommand{\LA}{\mathcal{L}_A}

\newcommand{\Rp}{R_\psi}
\newcommand{\Lp}{\mathcal{L}_\psi}

\newcommand{\thperp}{\theta^\perp}
\newcommand{\thoperp}{\theta_0^\perp}
\newcommand{\thbar}{\overline\theta}
\newcommand{\thp}{{(\theta,p)}}
\newcommand{\thpo}{{(\theta_0,p_0)}}

\newcommand{\norm}[1]{\left\lVert#1\right\rVert}      
\newcommand{\abs}[1]{\left|#1\right|}                 
\newcommand{\paren}[1]{\left(#1\right)}               
\newcommand{\bparen}[1]{\left[#1\right]}               
\newcommand{\sparen}[1]{\left\{#1\right\}}		      

\renewcommand{\epsilon}{\varepsilon}
\renewcommand{\rho}{\varrho}

\newcommand{\vp}{{\varphi}}
\newcommand{\vpp}{{(\varphi,p)}}
\newcommand{\vppo}{{(\varphi_0,p_0)}}

\newcommand{\tA}{\tilde{A}}
\newcommand{\tL}{\tilde{L}}

\newcommand{\lA}{\overline{A}}

\renewcommand{\tilde}{\widetilde}

\newcommand{\LD}{L^2(D)}
\newcommand{\Lloc}{L^2_{\text{loc}}}

\definecolor{jf}{rgb}{1,0,1}

\newcommand{\xb}{x_b}

\newcommand{\vpo}{\varphi_0}
\renewcommand{\th}{\theta}
\newcommand{\tho}{\theta_0}
\newcommand{\po}{p_0}

\newcommand{\So}{S^1}
\newcommand{\Sor}{S^1\times \mathbb{R}}
\renewcommand{\th}{\theta}

\usepackage{dsfont}
\newcommand{\one}{\mathds{1}}

\newcommand{\onea}{\one_{A}}

\newcommand{\rtwo}{{{\mathbb R}^2}}

\newcommand{\rn}{{{\mathbb R}^n}}

\newcommand{\nn}{\mathbb{N}}

\newcommand{\st}{\hskip 0.3mm : \hskip 0.3mm}

\newcommand{\be}{\begin{equation}}
\newcommand{\ee}{\end{equation}}
\newcommand{\bea}{\begin{eqnarray}}
\newcommand{\eea}{\end{eqnarray}}
\newcommand{\bean}{\begin{eqnarray*}}
\newcommand{\eean}{\end{eqnarray*}}

\newcommand{\bel}[1]{\begin{equation}\label{#1}}

\newcommand{\eel}[1]{{\label{#1}\end{equation}}}

\newcommand{\intt}{{\operatorname{int}}}
\newcommand{\ext}{{\operatorname{ext}}}
\newcommand{\cl}{{\operatorname{cl}}}
\newcommand{\bd}{{\operatorname{bd}}}

\newcommand{\smo}{\setminus\boldsymbol{0}}

\newcommand{\xio}{{\xi_0}}
\newcommand{\xo}{{x_0}}

\newcommand{\dx}{\mathrm{dx}}

\newcommand{\dpp}{\mathrm{dp}}

\newcommand{\dth}{\mathrm{d}\theta}
\newcommand{\dt}{\mathrm{d}t}

\newcommand{\omo}{\om_0}

\newcommand{\om}{\omega}

\newcommand{\etao}{\eta_0}

\graphicspath{{figures/}}

\numberwithin{equation}{section}

\title{Analyzing Reconstruction Artifacts from Arbitrary Incomplete
X-ray CT Data\thanks{Submitted to the editors DATE.  \funding{This
work was partially funded by: Innovation Fund Denmark and Maersk Oil
and Gas (Borg);  HC \O rsted Postdoc programme, co-funded by Marie Curie Actions at the Technical University of
Denmark (Frikel); ERC Advanced Grant No.\  291405 and EPSRC Grant
EP/P02226X/1 (J{\o}rgensen); Otto M{\o}nsteds Fond, Tufts FRAC and NSF
grants DMS 1311558 and DMS 1712207 (Quinto)}}}

\title[Artifacts from Arbitrary Incomplete CT Data]{Analyzing
Reconstruction Artifacts from Arbitrary Incomplete X-ray CT Data}

\author{Leise Borg$^1$}
  \author{J{\"u}rgen Frikel$^2$}
 \author{Jakob~Sauer J{\o}rgensen$^3$}
 \author{Eric~Todd Quinto$^4$}

\begin{document}

\maketitle 

\footnotetext[1]{Department of Computer Science, University of Copenhagen (lebo@di.ku.dk)}
\footnotetext[2]{Department of Computer Science and Mathematics,
OTH Regensburg, (juergen.frikel@oth-regensburg.de). Her
work was partially funded by: Innovation Fund Denmark and Maersk Oil
and Gas }

\footnotetext[3]{School of Mathematics, University of Manchester,
Manchester, M13 9PL, United Kingdom
(jakob.jorgensen@manchester.ac.uk). HC \O rsted Postdoc programme, co-funded by Marie Curie Actions at the Technical University of
Denmark (Frikel); ERC Advanced Grant No.\  291405 and EPSRC Grant
EP/P02226X/1 }

\footnotetext[4]{Department of Mathematics, Tufts University, Medford,
MA 02155 USA, (todd.quinto@tufts.edu) partial support from NSF grants
DMS 1311558 and 1712207, as well as support from the Otto M{\o}nsteds
Fond, DTU, and Tufts University Faculty Research Awards Committee}

 
\noindent\textbf{2010 AMS subject clasifications: }44A12, 92C55,
35S30, 58J40

%
%

\begin{abstract}
  This article provides a mathematical analysis of singular
    (nonsmooth) artifacts added to reconstructions by filtered
    backprojection (FBP) type algorithms for X-ray CT with arbitrary
  incomplete data.  We prove that these singular artifacts arise
  from points at the boundary of the data set.  Our results show that,
  depending on the geometry of this boundary, two types of artifacts can
  arise: object-dependent and object-independent artifacts.
  Object-dependent artifacts are generated by singularities of the object
  being scanned and these artifacts can extend along lines.  They
  generalize the streak artifacts observed in limited-angle
  tomography. Object-independent artifacts, on the other hand, are
  essentially independent of the object and take one of two forms: streaks
on lines if the boundary of the data set is not smooth at a point and
  curved artifacts if the boundary is smooth locally.  We prove that
    these streak and curve artifacts are the only singular artifacts that
    can occur for FBP in the continuous case. In addition to the geometric
  description of artifacts, the article provides characterizations of their
  strength in Sobolev scale in certain cases.  The results of this article
  apply to the well-known incomplete data problems, including limited-angle
  and region-of-interest tomography, as well as to unconventional X-ray CT
  imaging setups that  arise in new practical applications.
  Reconstructions from simulated and real data are analyzed to illustrate
  our theorems, including the reconstruction that motivated this work---a
  synchrotron data set in which artifacts appear on lines that have no
  relation to the object.

\end{abstract}

\date{\today}






\section{Introduction}\label{sect:intro}

Over the past decades computed tomography (CT) has established itself
as a standard imaging technique in many areas, including materials
science and medical imaging.  One collects X-ray measurements from
many different directions (lines) that are distributed all around the
object.  Then one reconstructs a picture of the interior of the object
 using an appropriate mathematical algorithm.  In classical
tomographic imaging setups, this procedure works very well because the
data can be collected all around the object, i.e., the data are
\emph{complete}, and standard reconstruction algorithms, such as filtered
backprojection (FBP), provide accurate reconstructions
\cite{Na:book,SidkyPanVannier2009}. However, in many CT problems, some
data are not available, and this leads to incomplete (or limited) data
sets. The reasons for data incompleteness might be patient related
(e.g., to decrease dose) or practical (e.g., when the scanner cannot
image all of the object, as in digital breast tomosynthesis).  

Classical incomplete data problems have been studied from the
beginning of tomography, including \emph{limited-angle tomography},
where the data can be collected only from certain view-angles
\cite{Louis:1980LA,Ka1997:limited-angle}; \emph{interior} or
\emph{region-of-interest (ROI) tomography}, where the X-ray
measurements are available only over lines intersecting a 
subregion of the object \cite{FRS, RZ, KR1996}; or \emph{exterior
tomography}, where measurements are available only over all lines outside a 
subregion \cite{Na1980, Q1998}.  

In addition, new scanning methods generate novel data sets, such
as the synchrotron experiment \cite{BJS-techreport,Borg2017} in
Section \ref{sect:synchrotron} that motivated this research. That
reconstruction, in Figure \ref{fig:streaks}, includes dramatic streaks
that are independent of the object and were not described in the
mathematical theory at that time but are explained by our main
theorems. A thorough practical investigation of this particular
problem was recently presented in \cite{Borg2017}.

\begin{figure}[t]
  \begin{center} \includegraphics[width=0.3\textwidth,
height=0.3\textwidth]{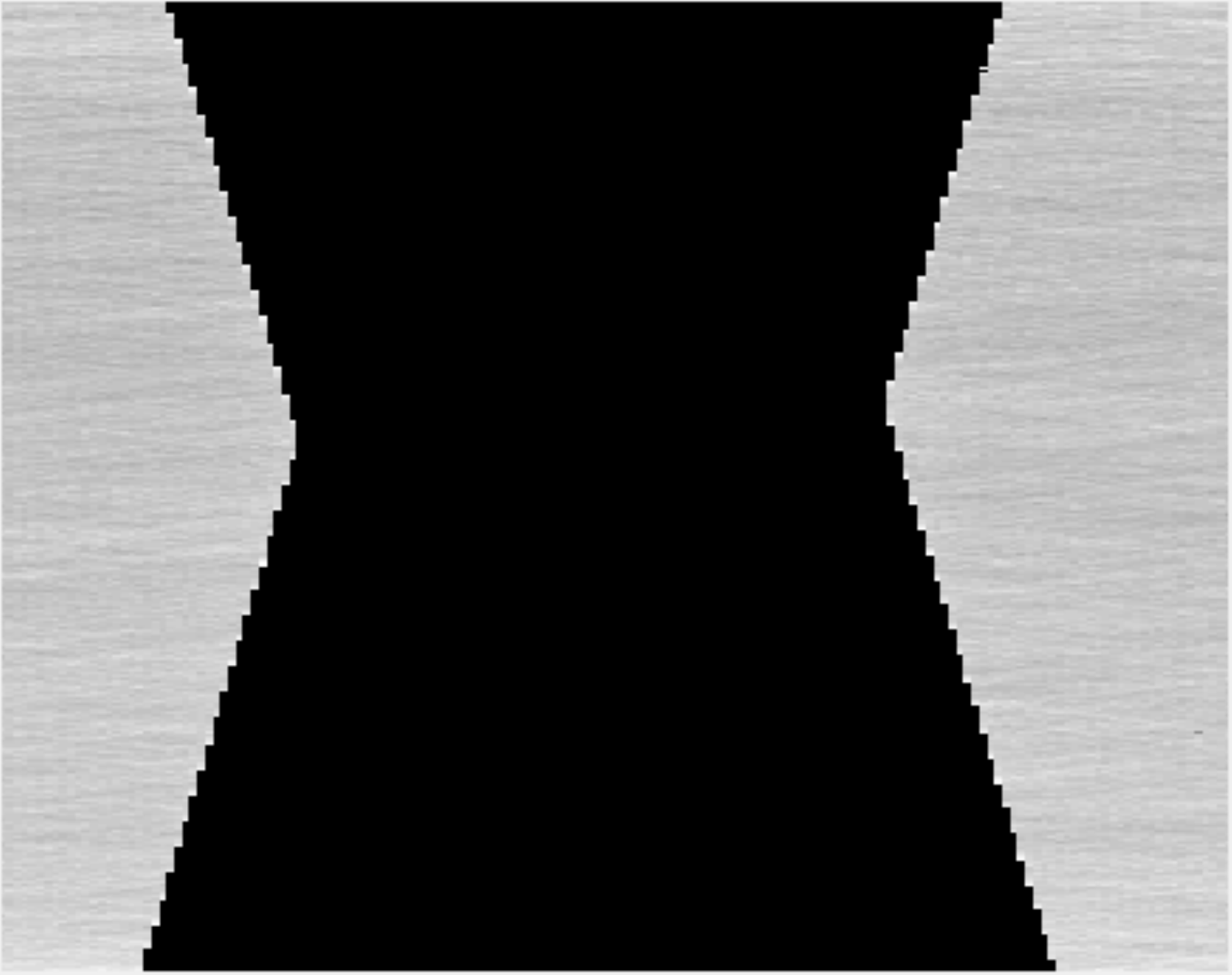} \quad
\includegraphics[width=0.3\textwidth,
height=0.3\textwidth]{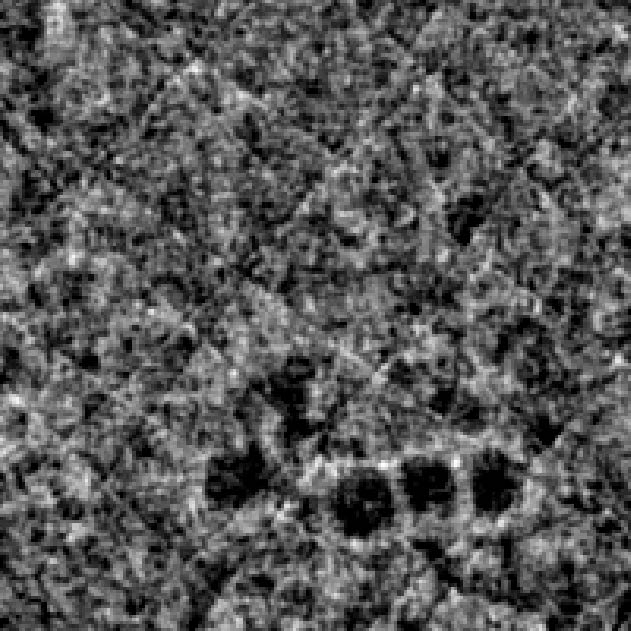}
\end{center}
\caption{\textit{Left:} A small part of the sinogram of the chalk
sample analyzed in Section \ref{sect:synchrotron}.  Notice that
boundary of the data set in this enlargement is jagged.
\textit{Right:} Small central section of a reconstruction of the
chalk.  Notice the streak artifacts over lines in the reconstruction.
Monochromatic parallel beam data were taken of the entire cross
section of the chalk over $1800$ views covering $180$ degrees, and
there were $2048\times 2048$ detector elements with a $0.5$ mm field
of view, providing micrometer resolution of the sample.  Data
\cite{Sorensen2016} obtained, with thanks from the Japan Synchrotron
Radiation Research Institute from beam time on beamline BL20XU of
SPring-8 (Proposal 2015A1147).  For more details, see Section
\ref{sect:synchrotron} and \cite[\copyright IOP Publishing.
Reproduced by permission of IOP Publishing.  All rights
reserved]{Borg2017}.}\label{fig:streaks}\end{figure}

Regardless of the type of data incompleteness, in most practical CT
problems a variant of FBP is used on the incomplete data to produce
reconstructions \cite{SidkyPanVannier2009}.  It is well-known that
incomplete data reconstruction problems that do not incorporate a
priori information (as is the case in all FBP type reconstructions)
are severely ill-posed (e.g., \cite{Lo1986} or \cite[Section
6]{Natterer01} for limited-angle CT).  Consequently, certain image
features cannot be reconstructed reliably \cite{Quinto93} and, in
general, artifacts, such as the limited-angle streaks in Figure
\ref{fig:LA-FBP} in Section \ref{sect:numerical} can occur.
Therefore, reconstruction quality suffers considerably, and this
complicates the proper interpretation of images. 

We  consider the continuous case, so we do not evaluate discretization
errors. By \emph{artifacts}, we mean nonsmooth image features
(singularities), such as streaks, that are added to the reconstruction by
the algorithm and are not part of the original object (see Definition
\ref{def:artifact+sing}).

\subsection{Related research in the mathematical literature} Our work
is based on microlocal analysis, a deep theory that describes how
singularities are transformed by Fourier integral operators, such as
the X-ray transform. Early articles using microlocal analysis in
tomography include \cite{Palamodov-1986}, which considers nonlinear
artifacts in X-ray CT, \cite{Quinto93}, which characterizes visible
and invisible singularities from X-ray CT data, \cite{GU1989} which
provided a general microlocal framework for admissible complexes, and
\cite{KLM} which considers general measures on lines in $\rtwo$.
Subsequently, artifacts were extensively studied in the context of
limited-angle tomography, e.g., \cite{Ka1997:limited-angle} and then
\cite{FrikelQuinto2013}. The strength of added artifacts in
limited-angle tomography was analyzed in \cite{Nguyen2015ip}.  Similar
characterizations of artifacts in limited-angle type reconstructions
have also been derived for the generalized Radon line and hyperplane
transforms as well as for other Radon transforms (such as circular and
spherical Radon transform), see
\cite{FrikelQuinto2015,FrikelQuinto-hyperplane, Nguyen-spherical,
Barannyk2016ip,Nguyen2017faa}.

Metal in objects can corrupt CT data and create dramatic streak
artifacts \cite{CT-Artifacts-BF:2012}.  This can be dealt with as an
incomplete data problem by excluding data over lines through the
metal.  Recently, this problem has been mathematically modeled in a
sophisticated way using microlocal analysis in \cite{ParkChoiSeo:2017,
Rigaud:2017,PUW-metal}.  A related problem is studied in
\cite{CPWWS-susceptibility, PalamodovSusceptibility-2016, PUW-streaks},
where the authors develop a streak reduction method for quantitative
susceptibility mapping.  Moreover, microlocal analysis has been used
to analyze properties of related integral transforms in pure and
applied settings \cite{BQ1987,GU1989, SU:SAR2013, FLU, Q2006:supp}.

\subsection{Basic mathematical setup and our results}

We use microlocal analysis to present a unified approach to analyze
reconstruction artifacts for arbitrary incomplete X-ray CT data that
are caused by the choice of data set.  We not only consider all of the
above mentioned classical incomplete data problems but also emerging
imaging situations with incomplete data.  We provide a geometric
characterization of the artifacts and we prove it describes all
singular artifacts that can occur for FBP type algorithms in the
continuous case. 


If $f$ is the density of the object to be reconstructed, then each CT
measurement is modeled by a line integral of $f$ over a line in the
data set.  As we will describe in Section \ref{sect:notation},
we parametrize lines by $\thp\in \Sor$, and the CT measurement of $f$
over the line $L\thp$ is denoted $Rf\thp$.  With complete data, where
$Rf\thp$ is given over all $\thp\in \Sor$, accurate reconstructions can 
be produced by the FBP algorithm.  In incomplete data CT
problems, the data are taken over lines $L\thp$ for $\thp$ in a proper
subset, $A$, of $\Sor$
and, even though FBP 
is designed for complete data, it is still one of the preferred 
reconstruction methods in practice, see \cite{SidkyPanVannier2009}. 
As a result, incomplete data CT reconstructions
usually suffer from artifacts. 


We prove that incomplete data artifacts arise from points at the
boundary or ``edge'' of the data set, $\bd(A)$, and we show that there
are two types of artifacts: \emph{object-dependent} and
\emph{object-independent} artifacts.  The object-dependent artifacts
are caused by singularities of the object being scanned.  In this
case, artifacts can appear all along a line $L\thpo$ (i.e., a streak)
if $\thpo\in \bd(A)$ \emph{and} if there is a singularity of the
object on the line (such as a jump or object boundary tangent to
the line)---this singularity of the object ``generates'' the artifact
(see Theorem \ref{thm:streak}\,\ref{thmPart:object dep}).
The streak artifacts observed in limited-angle tomography are special
cases of this type of artifact.

The object-independent artifacts are essentially independent of the
object being scanned (they depend primarily on the geometry of
$\bd(A)$) and they can appear either on lines or on curves.  If the
boundary of $A$ is smooth near a point $\thpo\in\bd(A)$, then we prove
that artifacts can appear in the reconstruction along curves generated
by $\bd(A)$ near $\thpo$, and they can occur whether the object being
scanned has singularities or not (see Theorem \ref{thm:curve}\,\ref{no
sing}\ref{Rf neq 0}). We also prove that, if $\bd(A)$ is not smooth
(see Definition \ref{def:smooth}) at a point $\thpo$, then,
essentially independently of the object, an artifact line can be
generated all along $L\thpo$ (see Theorem
\ref{thm:streak}\,\ref{thmPart:nonsmooth}). 

We will illustrate our results with reconstructions for classical
problems including limited-angle tomography and ROI tomography, as
well as problems with novel data sets, including the synchrotron data
set in Figure \ref{fig:streaks}.  In addition, we provide estimates of
strength of the artifacts in Sobolev scale.

To the best of our knowledge, the mathematical literature up
until now used microlocal and functional analysis to explain streak
artifacts on lines that are generated by singularities of the object,
\emph{and} they exclusively focused on specific problems, primarily
limited-angle tomography (e.g., \cite{Ka1997:limited-angle,
Nguyen2015ip, FrikelQuinto2013}).  Important work was done to analyze
visible singularities for ROI (or local) tomography (e.g., \cite{FRS,
Quinto93, RZ, KLM, KR1996}).  However, we are not aware of any
reference where a microlocal explanation for the ring artifact in ROI
CT was provided, although researchers are well aware of the ring
itself (e.g., \cite{EHMCCL-ringROI,CHrB-ringROI}).  We are also not
aware of microlocal analyses of more general imaging setups, such as
the nonstandard one presented in Figure \ref{fig:streaks}.

\subsection{Organization of the article}

In Section \ref{sect:math basis}, we provide notation and some of the basic
ideas about wavefront sets.  In Section \ref{sect:results} we give our main
theoretical results, and in Section \ref{sect:numerical}, we apply them to
explain added artifacts in reconstructions from  classical and novel limited
data sets.  In Section \ref{sect:strength}, we describe the strength of
added artifacts in Sobolev scale.  Then, in Section \ref{sect:reduction},
we describe a simple, known method to decrease the added artifacts and
provide a reconstruction and theorem to justify the method. We provide more
details of the synchrotron experiment in Section \ref{sect:synchrotron} and
observations and generalizations in Section \ref{sect:disc}.  Finally, in
the appendix, we give some technical theorems and then prove the main
theorems.

\section{Mathematical basis}\label{sect:math basis}

Much of our theory can be made rigorous for distributions of compact
support (see \cite{Friedlander98, Rudin:FA} for an overview of
distributions), but we will consider only Lebesgue measurable
functions.  This setup is realistic in practice, and our theorems are
simpler in this case than for general distributions.  Remark
\ref{rem:odds and ends} provides perspective on this.

The set $\LD$ is the set of square-integrable functions on the closed
unit disk\hfil\newline  $D=\sparen{x\in \rtwo\st \norm{x}\leq 1}$.  The set
$\Lloc(\rtwo)$ is the set of locally square-integrable
functions---functions that are square-integrable over every compact
subset of $\rtwo$.  We define $\Lloc(\Sor)$ in a similar way where
$S^1$ is the circle of unit vectors in $\rtwo$.

\subsection{Notation}\label{sect:notation}
 
Let $\thp\in \Sor$, then the line perpendicular to $\th$ and
containing $p\th$ is denoted \bel{def:L} L\thp = \sparen{x\in \rtwo\st
x\cdot \th = p}.\ee Note that $L\thp = L(-\th,-p)$.  For $\th\in \So$
let $\thperp$ be the unit vector $\pi/2$ radians counterclockwise from
$\th$.  We define the \emph{X-ray transform} or \emph{Radon line
transform} of $f\in \LD$ to be the integral of $f$ over $L\thp$:
\bel{def:R}Rf\thp =\int_{-\infty}^\infty f(p\th+t\thperp)\,\dt.\ee The
symmetry of our parametrization of lines gives the symmetry condition
\bel{symmetryR} Rf(\th,p) = Rf(-\th,-p).\ee For functions $g$ on
$\Sor$, the \emph{dual Radon transform} or \emph{backprojection
operator} is defined \bel{def:R*}R^* g(x) = \int_{\So}
g(\th,x\cdot\th)\,\dth.\ee When visualizing functions on $\Sor$, we
will use the natural identification \bel{def:thbar}\begin{aligned}
\rtwo\ni(\vp,p)&\mapsto (\thbar(\vp),p)\in \Sor\ \ \text{where}\ \
\thbar(\vp):= (\cos(\vp),\sin(\vp))\in \So
\end{aligned}\ee
and for functions $g$  on $\Sor$ the
identification 
\bel{def:tg}\tg(\vp,p) = g(\thbar(\vp),p) \text{ for } (\vp,p)\in \rtwo.\ee
The \emph{sinogram} of a function
$g\thp$ is a grayscale picture on $[0,\pi]\times \rr$ or
$[0,2\pi]\times \rr$ of the mapping $(\vp,p)\mapsto \tg(\vp,p)$.

\subsection{Wavefront sets}\label{sect:WF} In this section, we define
some important concepts needed to describe singularities in general.
Sources, such as \cite{Friedlander98}, provide introductions to
microlocal analysis.  Generally cotangent spaces are used to describe
microlocal ideas, but they would complicate this exposition, so we
will identify a covector $(x,\xi\dx)$ with the associated ordered pair
of vectors $(x,\xi)$.  The book chapter \cite{KrQu-chapter} provides
some basic microlocal ideas and a more elementary exposition adapted
for tomography.

The concept of the wavefront set is a central notion of microlocal
analysis. It defines singularities of functions in a way that
simultaneously provides information about their location and
direction. We will employ this concept to define (singular) artifacts
precisely, and we will use the powerful theory of microlocal analysis
to analyze artifacts generated in incomplete data reconstructions in
tomography.

In what follows, by a \emph{cutoff function at $\xo\in\rtwo$},
we will denote a $C^\infty$-function of compact support that is nonzero at
$\xo$. We now define singularities and the wavefront set.

\begin{definition}[Wavefront set  \cite{Treves:1980vf,Friedlander98}]
\label{def:WF} Let $\xo\in \rtwo$, $\xio\in \rtwo\smo$, and $f\in
\Lloc(\rtwo)$.  We say $f$ is \emph{smooth at $\xo$ in direction
$\xio$} if there is a cutoff function $\psi$ at $\xo$ and an open cone
$V$ containing $\xio$ such that the Fourier transform $\Fc(\psi
f)(\xi)$ is rapidly decaying at infinity for $\xi\in V$.\footnote{That
is, for every $k\in\nn$, there is a constant $C_k>0$ such that
$\abs{\Fc(\psi f)(\xi)}\leq C_k/(1+\norm{\xi})^k$ for all $\xi\in V$.}

We say  $f$ has a \emph{singularity at $\xo$ in direction $\xio$}, or a singularity at $(\xo,\xio)$, if $f$ 
is not smooth at $\xo$ in direction $\xio$. 

The \emph{wavefront set of $f$}, $\WF(f)$, is defined as the set of
all singularities $(\xo,\xio)$ of $f$.  \end{definition} 
$f$ has a \emph{singularity} at $\xo$ if $f$ is not smooth at $\xo$ in
some direction. 


For $(\xo,\xio)\in\WF(f)$, the first entry $\xo$ will be called the
\emph{base point} of $(\xo,\xio)$.  Hence, the base point of a
singularity gives the location where the function $f$ is singular (not
smooth) in some direction.  If we say $f$ has a singularity at $\xo$,
we mean $\xo$ is the base point of an element of $\WF(f)$.

As an example, let $B$ be a subset of the plane with a smooth boundary
and let $f$ be equal to $1$ on $B$ and $0$ off of $B$.  Then, $\WF(f)$
is the set of all points $(x,\xi)$ where the base points $x$ are on
the boundary of $B$ and $\xi$ is normal to the boundary of $B$ at $x$.
In this case, $f$ has singularities at all points of $\bd(B)$.

\begin{remark}[Wavefront set for functions defined on $\Sor$]
The notion of a singularity and the wavefront set can also be defined
for functions $g\in \Lloc(\Sor)$ using the identification
\eqref{def:tg}.

In order to define $\WF(g)$, let $\tg$ denote the locally
square-integrable function on $\rtwo$ defined by \eqref{def:tg}. Let
$\thp\in \Sor$ and $\vp\in \rr$ with $\th = \thbar(\vp)$. Let $\eta\in
\rtwo\smo$.  Then, we say that $g$ has a singularity at
$\paren{\thp,\eta}$ if $\tg$ has a singularity at
$\paren{(\vp,p),\eta)}$, i.e., $\paren{\thp,\eta}\in \WF(g)$ if
$\paren{(\vp,p),\eta}\in \WF(\tg)$. In that case, the base point of a
singularity of $g$ is of the form $\thp$.


Note that the wavefront set is well-defined for functions on $\Sor$ as
both $\tg$ and $\vp\mapsto \thbar(\vp)$ are $2\pi$-periodic in $\vp$.
\end{remark}





 \begin{definition}\label{def:WFsets} Let $(\th,p)\in \Sor$. The
\emph{normal space of the line $L\thp$} is \bel{NL}\NL =
\sparen{(x,\om\th)\st x\in L\thp,\, \om\in \rr}.\ee For $f\in
\Lloc(\rtwo)$, the set of singularities of $f$ normal to $L\thp$ is
\bel{WFL}\WF_{L(\th,p)}(f) = \WF(f)\cap \NL.\ee 

If $\WFL(f)\neq \emptyset$, then we say $f$ has a
\emph{singularity (or singularities) normal to $L\thp$}.

If $\WFL(f)=\emptyset$, then we say $f$ is \emph{smooth normal to the
line $L\thp$}.

For $\xo\in \rtwo$, we let \[\WF_{\xo}(f) = \WF(f)\cap
\paren{\sparen{\xo}\times \rtwo}.\]

For $g\in \Lloc(\Sor)$,  we define \bel{WFthp} \WF_{\thp}(g) = \WF(g)\cap
\paren{\sparen{\thp}\times \rtwo}.\ee 
\end{definition}

It is important to understand each set introduced in Definition
\ref{def:WFsets}: $\NL$ is the set of all $(x,\xi)$ such that $x\in
L\thp$ and the vector $\xi$ is normal to $L\thp$ at $x$.
Therefore, $\WF_{L\thp}(f)$ is the set of wavefront
directions $(x,\xi)\in\WF(f)$ with $x\in L\thp$ and $\xi$ normal to
this line.

The set $\WF_{\xo}(f)$ is the wavefront set of $f$ above $\xo$, and
$\WF_{\xo}(f)= \emptyset $ if and only if $f$ is smooth in some
neighborhood of $\xo$ \cite{Friedlander98}.

 If $g\in \Lloc(\Sor)$, then $\WF_{\thp}(g)$ is the set of wavefront
directions with base point $\thp$.  We will exploit the sets
introduced in these definitions starting in the next section.


\section{Main results}\label{sect:results}

In contrast to limited-angle characterizations in
\cite{Ka1997:limited-angle,FrikelQuinto2013}, our main results
describe artifacts in arbitrary incomplete data reconstructions
that include the classical limited data problems as special
cases.  Our results are formulated in terms of the wavefront
set (Definition \ref{def:WF}), which provides a precise concept of
singularity.

In many applications, reconstructions from incomplete CT data are
calculated by the filtered backprojection algorithm (FBP), which is
designed for complete data (see \cite{SidkyPanVannier2009} for a practical
discussion of FBP). In this case, the incomplete data are
often extended by the algorithm to a complete data set on $\Sor$ by
setting it to zero off of the set $A$ (\emph{cutoff region}) over which
data are taken. Therefore, the incomplete CT data can be modeled as
\bel{def:RA} \RA f\thp = \onea\thp Rf\thp, \ee where $\onea$ is the
characteristic function of $A$.\footnote{The characteristic function of a
  set $A$ is the function that is equal to one on $A$ and zero outside of
  $A$.}
Thus, using the FBP algorithm to calculate a reconstruction from such data
gives rise to the reconstruction operator:
\bel{def:LA} \LA f = R^*\paren{\Lambda \RA f} = R^*\paren{\Lambda
\onea R f}, \ee where $\Lambda$ is the standard FBP filter (see e.g.,
\cite[Theorem 2.5]{Na:book} and \cite[\S 5.1.1]{Natterer01} for
numerical implementations) and $R^*$ is defined by \eqref{def:R*}.

Our next assumption collects the conditions
we will impose on the cutoff region $A$. There, we will use the notation $\intt(A)$, $\bd(A)$, and
$\ext(A)$ to denote the interior of $A$, the boundary of $A$,
and the exterior of $A$, respectively.

\begin{assumption}\label{hyp:A} Let  $A$ be a  proper subset of  $\Sor$ 
(i.e., $A\neq  \Sor$) with a nontrivial interior and assume $A$
is symmetric in the following sense: \bel{A symmetry}\text{if }\thp\in
A\text{ then }(-\th,-p)\in A.\ee In addition, assume that $A$ is
the smallest closed set containing $\intt(A)$, i.e. $A=\cl(\intt(A))$.
\end{assumption}

We now explain the importance of this assumption. Since $A$ is proper,
data over $A$ are incomplete.  Being symmetric means that, if $\thp\in
A$ then the other parameterization of $L\thp$ is also in $A$.  We
exclude degenerate cases, such as when $A$ includes an isolated curve
by assuming that $A=\cl(\intt(A))$.

Our next definition gives us the language to describe the geometry of
$\bd(A)$.

\begin{definition}[Smoothness of $\bd(A)$]\label{def:smooth}
  Let $A\subset \Sor$ and let $\thpo\in \bd(A)$.  
\begin{itemize}

\item We say that $\bd(A)$ is \emph{smooth near $\thpo$} if, for some
neighborhood, $U$ of $\thpo$ in $\Sor$, the part of $\bd(A)$ in $U$ is
a $C^\infty$ curve.  In this case, there is a unique tangent line in
$(\th,p)$-space to $\bd(A)$ at $\thpo$.
\begin{itemize}
\item If this tangent line is vertical (i.e., of the form $\th=\tho$),
then we say the boundary \emph{is vertical} or \emph{has infinite
slope} at $\thpo$.

\item If this tangent line is not vertical, then $\bd(A)$ is defined
near $\thpo$ by a smooth function $p=p(\th)$.  In this case, the
\emph{slope of the boundary at $\thpo$} will be the slope of this
tangent line:
\bel{def:p'}p'(\tho):=\frac{dp}{d\vp}\paren{\thbar(\vpo)}\ \
\text{where $\vpo$ is defined by $\thbar(\vpo) =
\tho$.}\footnote{Note that the map $\vp\mapsto \thbar(\vp)$
gives the local coordinates on $\So$ near $\vpo$ and $\tho$ that are
used in our proofs, and $p'$ is just the derivative of $p$ in these
coordinates.}\ee
\end{itemize}

\item We say that $\bd(A)$ is \emph{not smooth at $\thpo$} if it is
not a smooth curve in any neighborhood of $\thpo$.

\begin{itemize}
\item We say that $\bd(A)$ has \emph{a corner at $\thpo$} if the curve
$\bd(A)$ is continuous at $\thpo$, is smooth at all other points
sufficiently close to  $\thpo$, and has one-sided tangent lines at
$\thpo$ but they are different
lines.\footnote{\label{footnote:corner}Precisely, there
is an open neighborhood $U$ of $\thpo$, an open interval $I=(a,b)$,
 two smooth functions $c_i:I\to
U$, $i=1,2$, and some $t_0\in I$ such that $c_i(t_0)=\thpo,\ i=1,2$; the
curves $c_1(I)$ and $c_2(I)$ intersect transversally at $\thpo$; and
$\bd(A)\cap U = c_1((a,t_0])\cup c_2((a,t_0])$.}
\end{itemize}
\end{itemize}
\end{definition}

\subsection{Singularities and artifacts}\label{sect:visible} In this
section we define artifacts and visible and invisible singularities, and we
explain why artifacts appear on lines $L\thp$ only when $\thp\in \bd(A)$.

 \begin{definition}[Artifacts and visible
singularities]\label{def:artifact+sing} Every singularity $(x,\xi)\in
   \WF(\LA f)$ that is not a singularity of $f$ is called an
   \emph{artifact} (i.e., any singularity in $\WF(\LA f)\setminus \WF(f)$).

   An \emph{artifact curve} is a collection of base points of
artifacts that form a curve.

A \emph{streak artifact}
is an artifact curve in which the curve is a subset of a line.

Every singularity of $f$ that is also in $\WF(\LA f)$ is said to be
\emph{visible} (from data on $A$), i.e., any singularity in $\WF(\LA f)\cap
\WF(f)$. Other singularities of $f$ are called \emph{invisible} (from data
on $A$).\footnote{Invisible singularities of $f$ are smoothed by $\LA$ and
  reconstruction of those singularities is in general extremely ill-posed
  in Sobolev scale since any inverse operator must take each smoothed
  singularity back to the original non-smooth singularity, so inversion
  would be discontinuous in any range of Sobolev norms.}
\end{definition}

Our next theorem gives an analysis of singularities in $\LA f$
corresponding to lines $L\thp$ for $\thp\notin\bd(A)$. It
shows that the only singularities of $\LA f$ that are normal to lines
$L\thp$ for $\thp\in \intt(A)$ are visible singularities of $f$, and
there are no singularities of $\LA f$ normal to lines $L\thp$ for
$\thp\in \ext(A)$.

\begin{theorem}[Visible and invisible singularities in
the reconstruction] \label{thm:visible} Let $f\in \LD$ and let
$A\subset \Sor$ satisfy Assumption \ref{hyp:A}.
\begin{enumerate}[label=\Alph*.]
\item\label{thm:int} If $\thp\in \intt(A)$ then $\WFL(f)= \WFL(\LA
f)$. Therefore, all singularities of $f$ normal to $L\thp$ are
visible singularities, and $\LA f$ has no artifacts normal to
$L\thp$.

  \item\label{thm:ext} If $\thp\notin \paren{A\cap \supp(Rf)}$, then
$\WFL(\LA f)=\emptyset$.  Therefore, all singularities of $f$
normal to $L\thp$ are invisible from data on $A$,and $\LA f$ has no
artifacts normal to $L\thp$.

\item\label{thm:all} If $x\in D$ and all lines through $x$ are
parameterized by points in $\intt(A)$ (i.e., $\forall \th\in \So$,
$(\th,x\cdot\th)\in \intt(A)$), then \bel{visible WF}
\WF_{x}(f)= \WF_x(\LA f).\ee In this case, all singularities of $f$ at
$x$ are visible in $\LA f$.

\end{enumerate}
Therefore, artifacts occur only normal to lines $L\thp$ for
$\thp\in\bd(A)$.
\end{theorem}

This theorem follows directly from \cite[Theorem 3.1]{Quinto93}
and continuity of $R^*$ (see also \cite{KLM}).  Note that Theorem
\ref{thm:visible}\,\ref{thm:all} follows from parts \ref{thm:int} and
\ref{thm:ext} and is included because we will need it later.

\subsection{Analyzing singular
artifacts}\label{sect:characterization} 

We now analyze artifacts in limited data FBP reconstructions using
$\LA$ \eqref{def:LA}.  In particular, we show that the nature of
artifacts depends on the smoothness and geometry of $\bd(A)$ and, in
some cases, singularities of the object $f$.  

Theorem \ref{thm:visible} establishes that artifacts occur only above
points on lines $L\thp$ for $\thp\in \bd(A)$.  Our next two theorems
show that the only artifacts that occur are either artifacts on
specific types of curves (see \eqref{def:xb}) or streak artifacts, and
they are of two types.  

Let $f\in \LD$ and let $\thp\in \bd(A)$:
\begin{itemize}
\setlength{\itemsep}{2mm} \item \textbf{Object-independent artifacts:}
those are caused essentially by the geometry of $\bd(A)$.  They can
occur whether $f$ has singularities normal to $L\thp$ or not, and they
can be curves or streak artifacts. 

\item \textbf{Object-dependent artifacts:} those are caused
essentially by singularities of the object $f$ that are normal to
$L\thp$.  They  will not occur if $f$ is smooth normal to $L\thp$, and
they are always streak artifacts.

\end{itemize}

Our next theorem gives conditions under which artifact curves that are
not streaks (i.e., not subsets of lines) appear in reconstructions
from $\LA$. 

\begin{theorem}[Artifact Curves]\label{thm:curve}
 Let $f\in \LD$ and let $A\subset \Sor$ satisfy Assumption \ref{hyp:A}.
Let $\thpo\in \bd(A)$ and assume that $\bd(A)$ is smooth near $\thpo$.
Assume $\bd(A)$ has finite slope at $\thpo$ and let $I$ be a
neighborhood of $\tho$ in $\So$ such that $\bd(A)$ is given by a
smooth curve $p=p(\th)$ near $\thpo$.  Let
\begin{equation} \label{def:xb} \xb =\xb(\th) =p(\th)\th +
p'(\th) \thperp\in \rtwo\ \ \text{for $\th\in I$.}\end{equation} Then,
an object-independent artifact curve can appear in $\LA f$ on the
curve given by $I\ni \th\mapsto \xb(\th)$, which we will call the
\emph{$\xb$-curve}.

\begin{enumerate}[label=\Alph*.]
\item\label{not a line} The $\xb$-curve is curved (i.e., not a subset
of a line) unless it is a point.

\item\label{no sing} Assume $f$ is smooth normal to $L\thpo$.  
\begin{enumerate}[label=(\arabic*)]

\item\label{no sing:WF point} Then, \bel{sect3:WF point} \WFLo(\LA f)
\subset\sparen{\paren{\xb(\tho),\omega\tho}\st \omega\neq 0}.\ee

\item\label{Rf=0} If $Rf=0$ in a neighborhood of $\thpo$, then
$\WFLo(\LA f) = \emptyset$ and this $\xb$-curve will not
appear in the reconstruction $\LA f$ near $\xb(\tho)$.
 
 \item\label{Rf neq 0} If $Rf\thpo\neq 0$, then equality holds
in \eqref{sect3:WF point} and the $\xb$-curve will
appear in the reconstruction $\LA f$ near $\xb(\tho)$. 
\end{enumerate}
\end{enumerate}
\end{theorem}

 Theorem \ref{thm:curve} is proven in Appendix \ref{appendix:added}.
 Figures \ref{fig:sqrt}, \ref{fig:ROI}, and \ref{fig:general artifacts} in
 Section \ref{sect:numerical} all show $\xb$-artifact curves.  The
 following remark discusses these curves in more detail.

 \begin{remark}\label{rem:xb and D} Assume $\bd(A)$ is smooth with
finite slope at $\thpo$.  Let $I$ be a neighborhood of $\tho$ and let
$p:I\to\rr$ be a parametrization of $\bd(A)$ near $\thpo$.  Note that 
\[\xb(\th) \in L\thp\ \text{ for }\ \th\in I.\]

  If the slope of $\bd(A)$ at $\thpo$ is small enough, i.e., \bel{small
    slope}\abs{p'(\tho)}< \sqrt{1-p_0^2}\ee holds, then the $\xb$-curve of
  artifacts $\th\mapsto \xb(\th)$ will be inside the closed unit
  disk, $D$,
  at least for $\th$ near $\tho$.  If not, then $\xb(\tho)\notin
  \intt(D)$. This is illustrated in Section \ref{sect:numerical} in Figure
  \ref{fig:sqrt large} for large slope-- where \eqref{small slope} is not
  satisfied, and \ref{fig:sqrt small} for small slope--where \eqref{small
    slope} is satisfied.

If $\bd(A)$ is smooth and vertical at $\thpo$ (infinite slope), then there
will be no object-independent artifact on the line $L\thpo$.  This follows
from the proof of this theorem because the singularity in the data that
causes the $\xb$ curve is smoothed by $R^*$ in this case.  Intuitively, if
$\bd(A)$ is vertical then $p'(\tho)$ is infinite and from \eqref{def:xb},
the point $\xb(\tho)$ would be ``at infinity.''  In this case, only
object-dependent streak artifacts can be generated by $\thpo$, see Theorem
\ref{thm:streak} and Figures \ref{fig:LA-FBP} and \ref{fig:sqrt} in Section
\ref{sect:numerical}.
\end{remark}

Our next theorem gives the conditions under which there can
be streak artifacts in reconstructions using $\LA$.

 \begin{theorem}[Streak artifacts]\label{thm:streak} Let $f\in \LD$
and let $A\subset \Sor$ satisfy Assumption \ref{hyp:A}.

\begin{enumerate}[label=\Alph*.]

\item\label{thmPart:object dep} If $f$ has a singularity normal
to $L\thpo$, then a streak artifact can occur on $L\thpo$.

\item \label{thmPart:smooth object depII} If $f$ is smooth normal to
$L\thpo$ and $\bd(A)$ is smooth and vertical at $\thpo$, then $\LA f$
is smooth normal to $L\thpo$.\footnote{Note that Theorem
\ref{thm:curve}\,\ref{no sing} states that, if $f$ is smooth normal to
$L\thpo$ and $\bd(A)$ is smooth and not vertical at $\thpo$, then $\LA
f$ is smooth normal to $L\thpo$ except possibly at $\xb(\tho)$ (see
\eqref{sect3:WF point}).}

\item 
\label{thmPart:nonsmooth} Let $\thpo\in \bd(A)$ and assume that
$\bd(A)$ is not smooth at $\thpo$.  Then, $\LA f$ can have a streak
artifact on $L\thpo$ independent of $f$.\smallskip

If $f$ is smooth normal to $L\thpo$, then $Rf\thpo\neq 0$, and
$\bd(A)$ has a corner at $\thpo$ (see Definition \ref{def:smooth}),
then $\LA f$ \emph{does} have a streak artifact on $L\thpo$, i.e.,
\[\WFLo(\LA f) = \NLo.\]
\end{enumerate}
\end{theorem}

The proof Theorem \ref{thm:streak} is provided in  Appendix \ref{appendix:added}.

Part \ref{thmPart:object dep} of Theorem \ref{thm:streak} provides a
generalization of classical limited-angle streak artifacts observed in
Figure \ref{fig:LA-FBP} in Section \ref{sect:numerical}. Such
limited-angle type artifacts can also be seen in Figures
\ref{fig:sqrt} and \ref{fig:general artifacts} in that section.

Part \ref{thmPart:smooth object depII} of Theorem \ref{thm:streak}
shows that the streak artifacts in Part \ref{thmPart:object
dep} are object-dependent.

Part \ref{thmPart:nonsmooth} of Theorem \ref{thm:streak} explains the
object-independent streak artifacts in Figure \ref{fig:general
artifacts} that are highlighted in
yellow as well as the object-independent streak artifacts that are
observed in the real data reconstructions in Figures \ref{subfig:sync
data fbp rec} and \ref{subfig:sync data fbp rec} in Section
\ref{sect:synchrotron}.  In Theorem \ref{thm:Sobolev sing}, we will
describe the strength of the artifacts in Sobolev scale in specific
cases of Theorems \ref{thm:curve} and \ref{thm:streak}.

 \begin{example}\label{ex:no artifacts} Theorem \ref{thm:curve} and
Theorem \ref{thm:streak} give necessary conditions under which $\LA f$ can
have artifacts.  We now provide an example when the conditions of those
theorems hold for $f$ and $A$ but $\LA f$ has no artifacts.  This is why we
state in parts of Theorems \ref{thm:curve} and \ref{thm:streak} that
artifacts \emph{can} occur, rather than that they \emph{will} occur.

 Let $A=\sparen{\thp\in \Sor \st \abs{p}\leq 1}$, then $A$ represents
the set of lines meeting the closed unit disk, $D$.  Let $f$ be the
characteristic function of $D$.  Then, for all $x\in \bd(D)=S^1$, $\xi
= (x,x)\in \WF(f)$, $\xi$ is normal to the line $L(x,1)$, and
$(x,1)$, which is in $\Sor$, is also in $ \bd(A)$.  Under these
conditions, there \emph{could} be a streak artifact on $L(x,1)$ by
Theorem \ref{thm:streak}\,\ref{thmPart:object dep} Because $\bd(A)$ is
smooth and not vertical, there \emph{could} be an $\xb$-curve artifact
by Theorem \ref{thm:curve}.  However, $\onea Rf = Rf$ so $\LA f = f$
and there are no artifacts in this reconstruction.
\end{example}

Object-dependent streak artifacts were analyzed for limited-angle
tomography in articles such as \cite{Ka1997:limited-angle,FrikelQuinto2013,
  Nguyen2015ip}, but we are unaware of a reference to Theorem
\ref{thm:streak}\,\ref{thmPart:object dep} for general incomplete data
problems.  We are not aware of a previous reference in the literature to a
microlocal analysis of the $\xb$-curve artifact as in Theorem
\ref{thm:curve} or to the corner artifacts as in Theorem
\ref{thm:streak}\,\ref{thmPart:nonsmooth} We now assert that all singular
artifacts are classified by Theorems \ref{thm:curve} and \ref{thm:streak}.

\begin{theorem}\label{thm:complete} Let $f\in \LD$
and let $A\subset \Sor$ satisfy Assumption \ref{hyp:A}.  The only singular
artifacts in $\LA f$ occur on $\xb$-curves as described by Theorem
\ref{thm:curve} or are streak artifacts as described by Theorem
\ref{thm:streak}.\end{theorem}

 Theorem \ref{thm:complete} is proven in Section \ref{appendix:added}.

\section{Numerical illustrations of  our theoretical
results}\label{sect:numerical}

We now consider a range of well-known incomplete data problems as well
as unconventional ones to show how the theoretical results in Section
\ref{sect:results} are reflected in practice.  All sinograms represent
the data $g\thp=Rf\thp$ using \eqref{def:tg} and displaying them in
the $\vpp$-plane rather than showing them on $\Sor$.  To this end, we
define \bel{otpr-conventions}\begin{gathered} \tL\vpp :=
L(\thbar(\vp),p),\\
(\vp,p)\mapsto \tg(\vp,p)=g(\thbar(\vp),p)\ \text{ for }
\vp\in [0,2\pi], \ p\in [-\sqrt{2},\sqrt{2}],\\
\text{if $A\subset\Sor$, then $\tA := \sparen{\vpp\in [0,2\pi]\times
\rr\st (\thbar(\vp),p)\in A}$.}\end{gathered}\ee In this section, we
will specify limited data using the sets $\tA\subset [0,2\pi]\times
\rr$ rather than $A\subset \Sor$, and  we will let
$R$ denote the Radon transform with this parametrization.
Furthermore, because of the symmetry condition \eqref{symmetryR}, we
will display only the part of the sinogram in $[0,\pi]\times
\bparen{-\sqrt{2},\sqrt{2}}$.  Except for the center picture in Figure
\ref{fig:sqrt large}, reconstructions are displayed on $[-1,1]^2$.

\subsection{Limited-angle tomography}\label{sect:LimitedAngle} 

First, we analyze limited-angle tomography, a classical problem in
which Theorem \ref{thm:streak}\,\ref{thmPart:object dep}
applies.  In this case $\bd(\tA)$ consists of four vertical lines $\vp
= \vp_1$, $\vp=\vp_2$, $\vp = \vp_1+\pi$, $\vp=\vp_2+\pi$ for two
angles $0\leq \vp_1<\vp_2<\pi$ representing the ends of the angular
range. Taking a closer look at the statement of
Theorem \ref{thm:streak}\,\ref{thmPart:object dep} and the
results of \cite{FrikelQuinto2013, FrikelQuinto-hyperplane} one can
observe that, locally, they describe the same phenomena, namely:
whenever there is a line $\tL\vppo$ in the data set with $\vppo\in
\bd(\tA)$ and which is normal to a singularity of $f$, then a streak
artifact can be generated on $\tL\vppo$ in the reconstruction $\LA
f$.  Therefore, Theorem \ref{thm:streak}\,\ref{thmPart:object
dep} generalizes the results of \cite{Ka1997:limited-angle,
FrikelQuinto2013} as it also applies to cutoff regions with
non-vertical tangent.

It is important to note that, with limited-angle data, there are no
object-independent artifacts since $\bd(\tA)$ is smooth and vertical
(the $\xb$-curve is not defined).  
\begin{figure}[h!]
\newcommand{\ww}{0.3\linewidth} \centering
  \includegraphics[height=\ww]{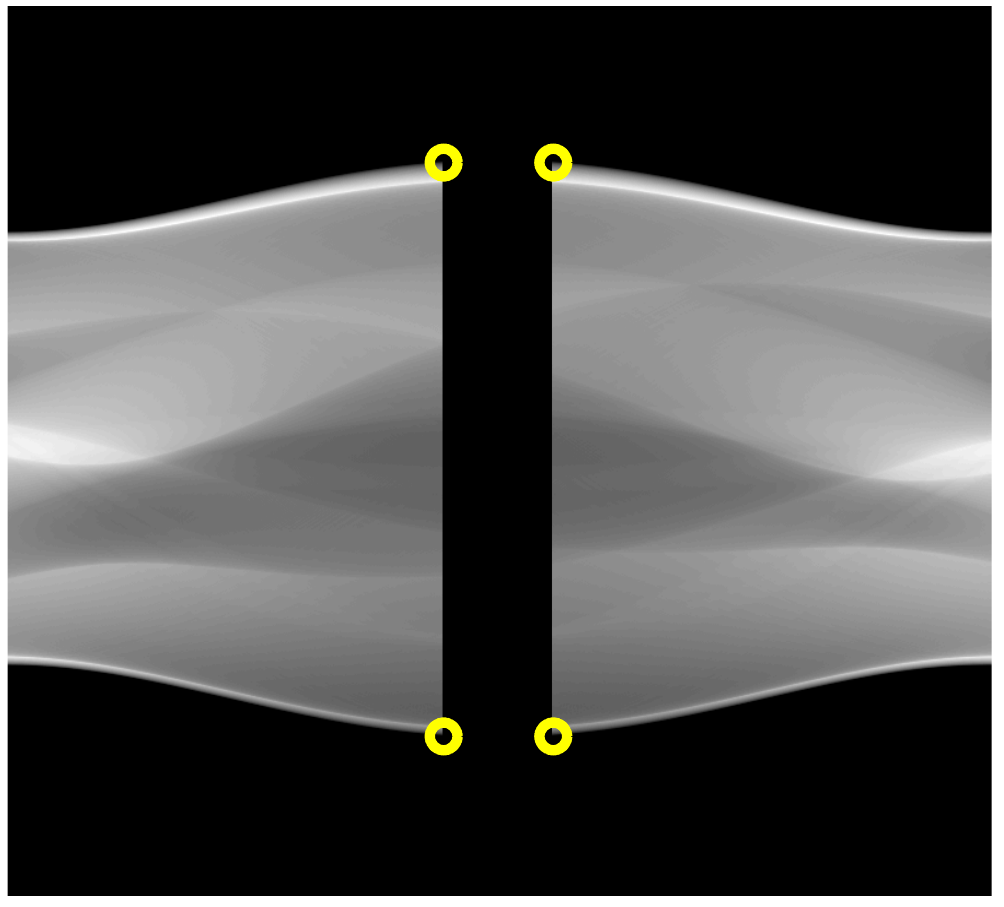}
\includegraphics[height=\ww]{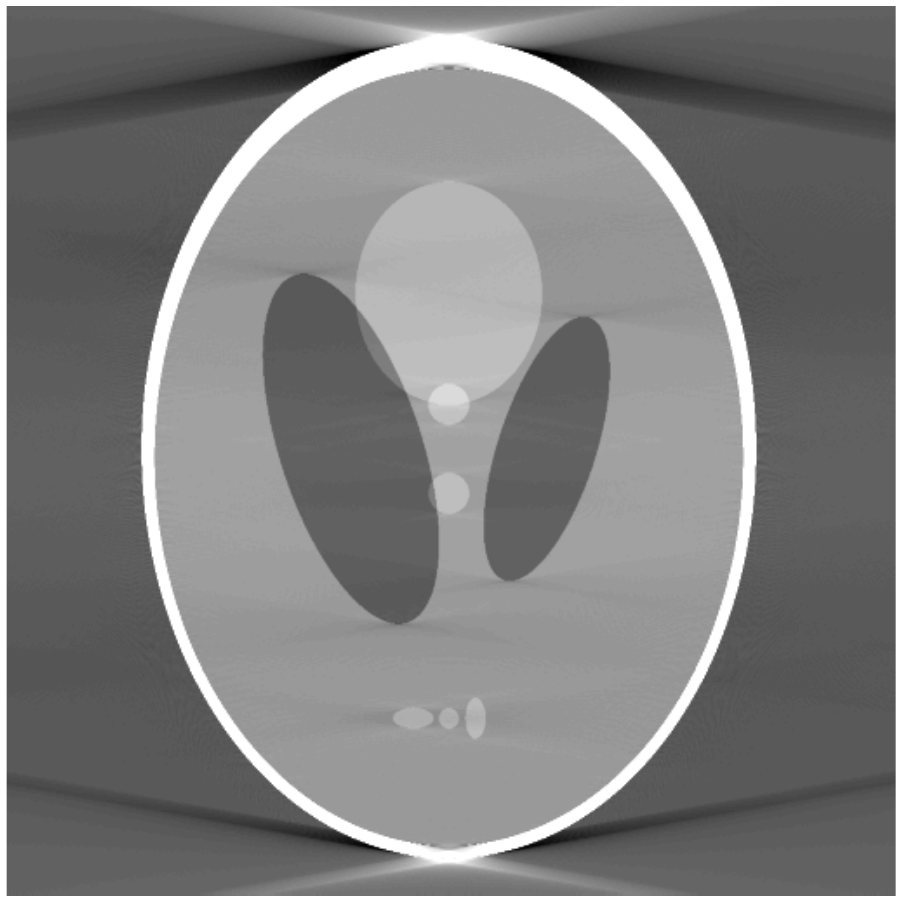}
\includegraphics[height=\ww]{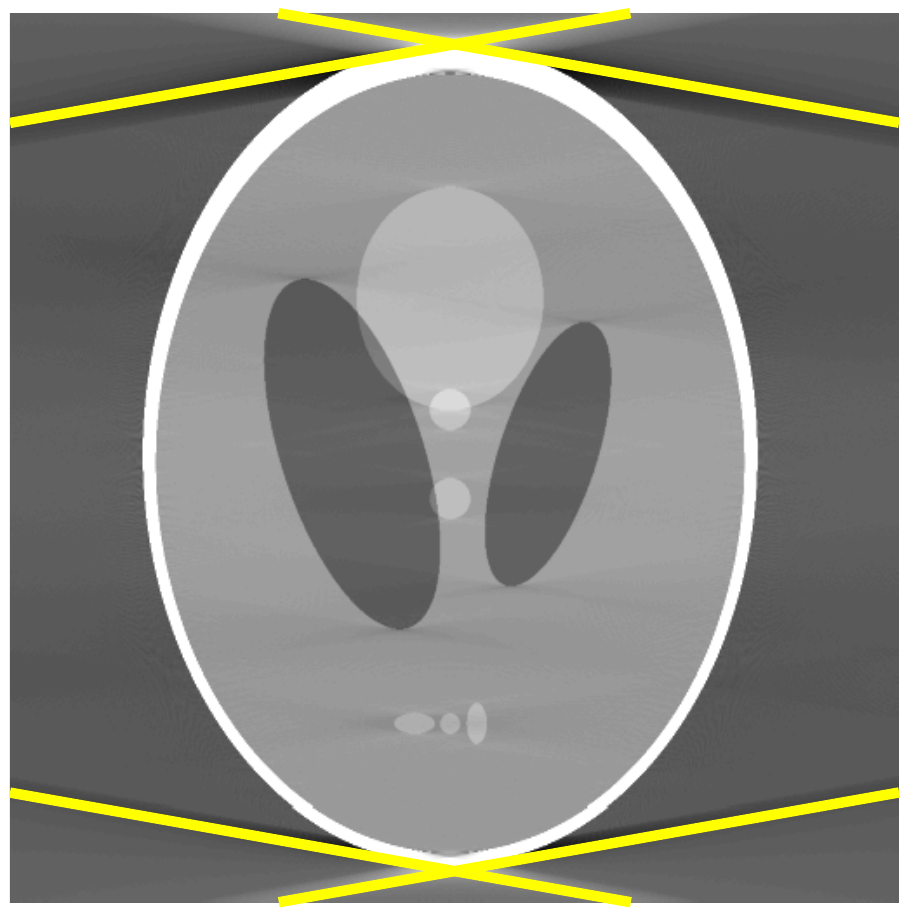}
\caption{\textit{Left:} Limited-angle data ($\bd(\tA)$ is
vertical). \textit{Center:} FBP reconstruction. \textit{Right:}
Reconstruction highlighting object-dependent artifact lines tangent to
skull corresponding to the four circled points in the sinogram. }
\label{fig:LA-FBP}
   \end{figure} 

Figure \ref{fig:LA-FBP} illustrates limited-angle tomography. The
boundary, $\bd(\tA)$,
consists of the vertical lines $\vp = 4\pi/9$ and $\vp=5\pi/9$.  The
artifact lines are exactly the lines with $\vp=4\pi/9$ or $5\pi/9$
that are tangent to boundaries in the object (i.e., wavefront
directions are normal to the line).  The four circled points on the
sinogram correspond to the object-dependent artifact lines at the
boundary of the skull.  The corresponding lines are tangent to the
skull and have angles $\vp=4\pi/9$ and $\vp = 5\pi/9$. One can also
observe artifact lines tangent to the inside of the skull with these
same angles. 

One can notice invisible singularities of $f$---the top and bottom
boundaries of the skull---at the top and bottom of the reconstruction.
If the excluded region were larger, they would be more noticeable.

\subsection{Smooth boundary with finite slope}\label{sect:smooth} 

We now consider the general case in Theorem \ref{thm:curve} by
analyzing the artifacts for a specific set $\tA$ which is defined as
follows.  It will be cut in the middle so that the left-most boundary
of $A$ occurs at $\vp = a: = \frac{4}{9}\pi$; the right-most boundary
is constructed as $\vp = b: = \frac{5}{9}\pi$ for $p \leq 0$ and 
\begin{equation} \label{eq:p_sqrt}
p(\vp) = c \sqrt{\vp - b}, \quad \vp > b
\end{equation}
for $p>0$ such that the two parts join differentiably at $\vpp =
(0,0)$. The steepness of the curved part of the right-most boundary is
governed by the constant $c$ (as seen in the two sinograms in Figure
\ref{fig:sqrt}).

 \begin{figure}[h!] \begin{subfigure}[t]{0.98\textwidth}
\newcommand{\ww}{0.3\linewidth} \centering
\includegraphics[height=\ww]{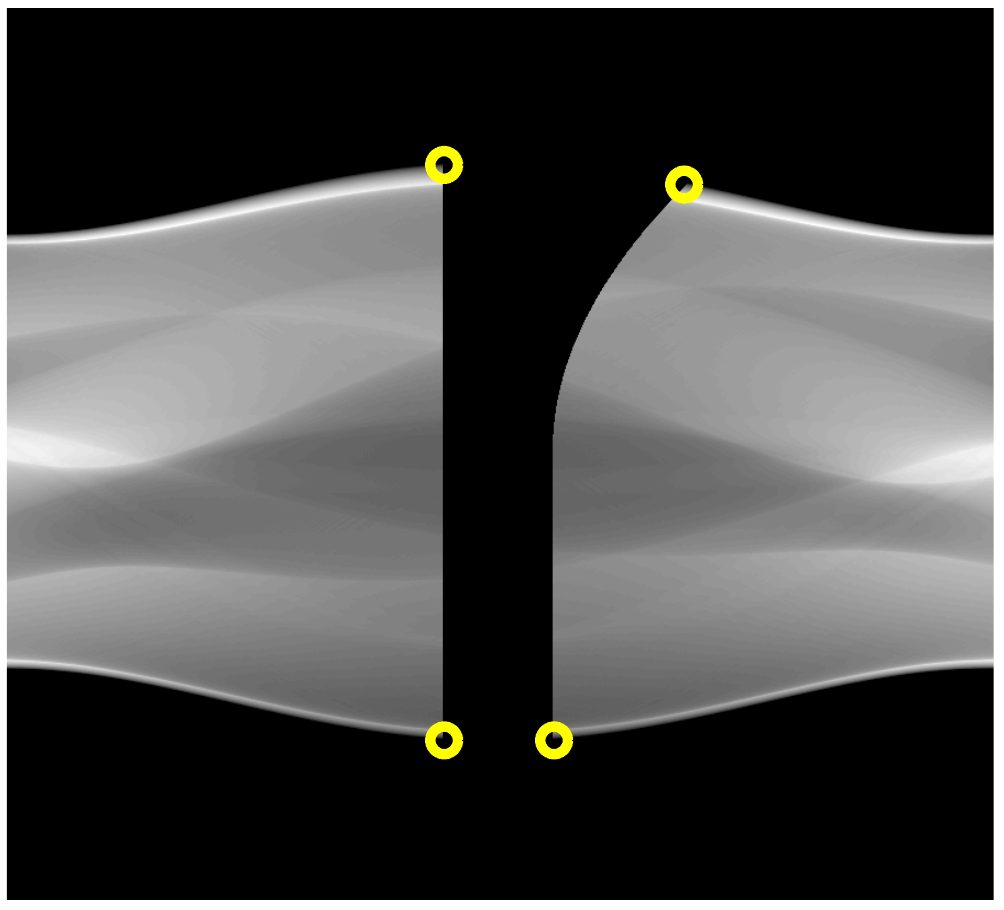}
\includegraphics[height=\ww]{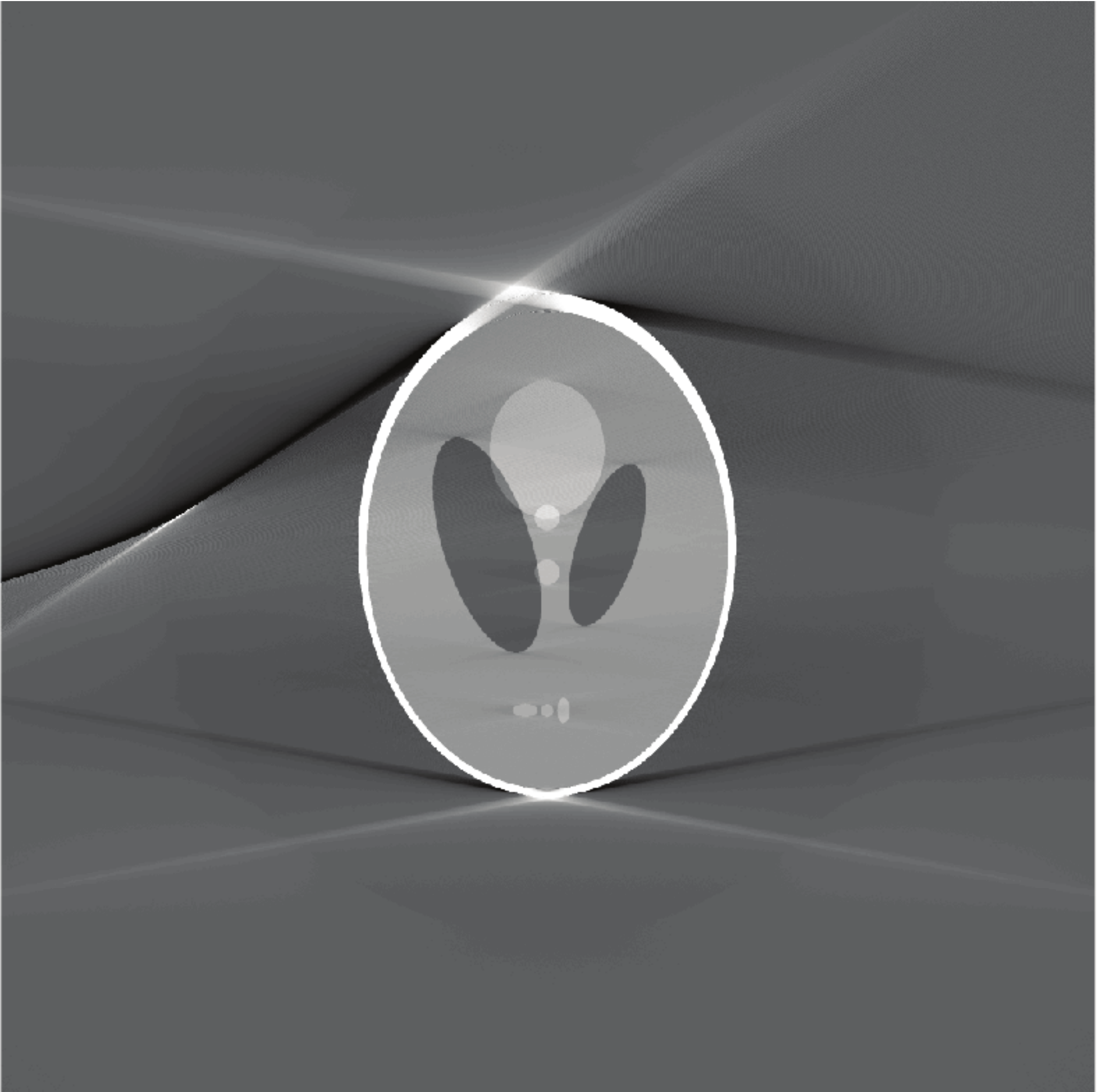}
\includegraphics[height=\ww]{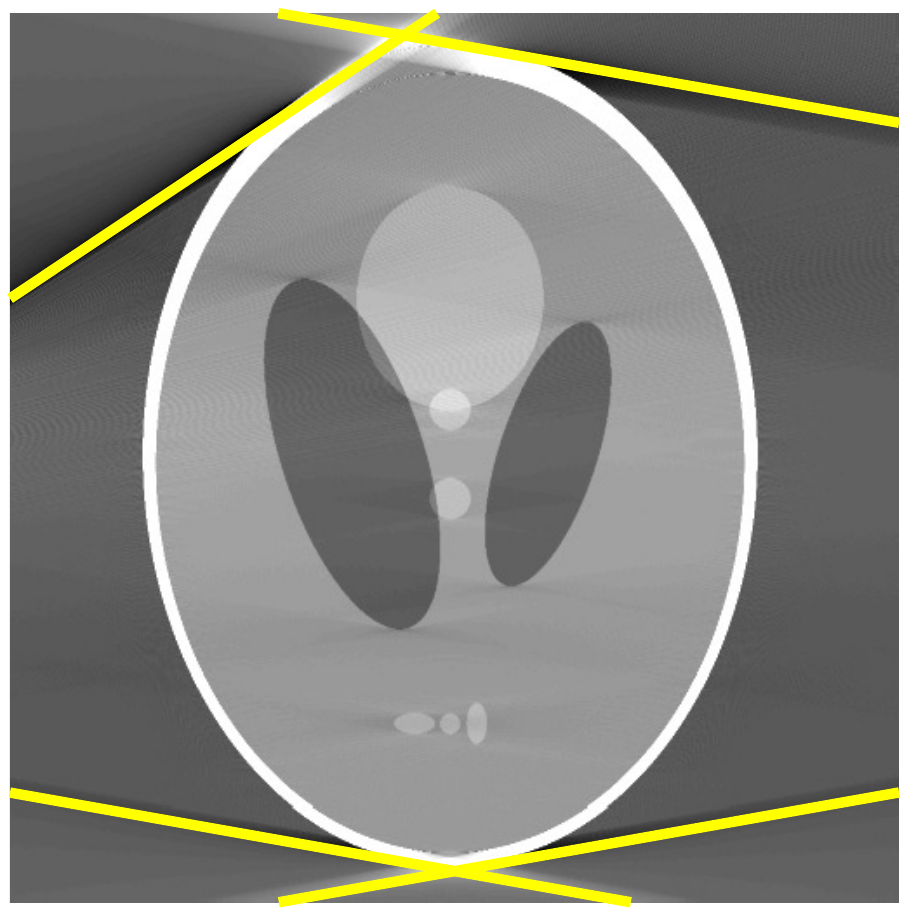}
\caption{\textit{Left:} Sinogram with the boundary of $\tA$ having
large slope ($c=1.3$).  \textit{Center:} FBP reconstruction over the
larger region $[-2,2]^2$ to show that the $\xb$-curve of
artifacts is outside of the region displayed in the right frame.
\textit{Right:} Reconstruction highlighting object-dependent artifact
lines tangent to the skull corresponding to the four circled points in
the sinogram.} \label{fig:sqrt large}
\end{subfigure}
\newline
\begin{subfigure}[t]{1.0\textwidth}\centering
\newcommand{\ww}{0.3\linewidth}
  \includegraphics[height=\ww]{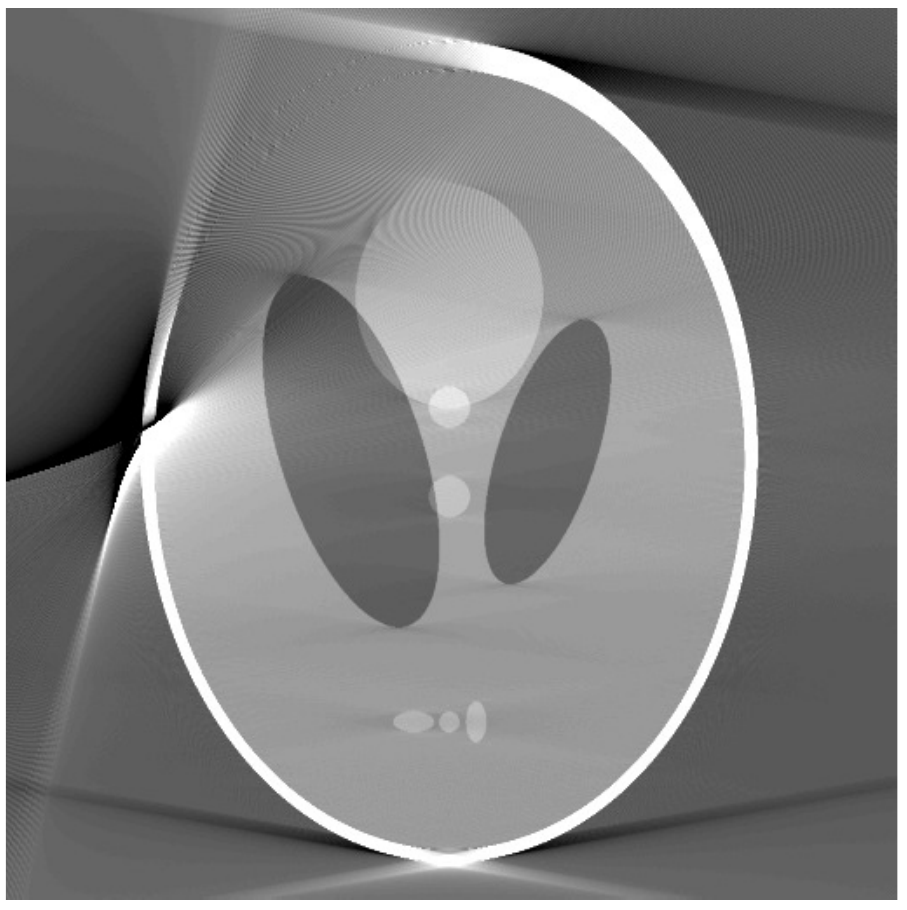}
\includegraphics[height=\ww]{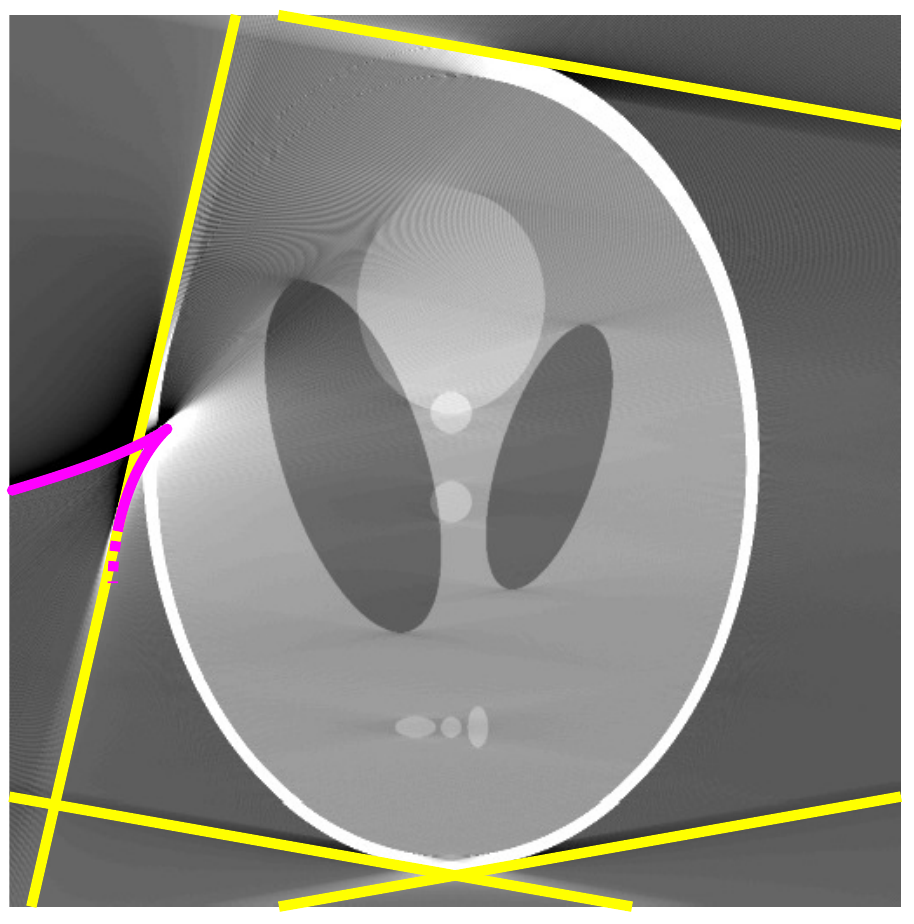}
\caption{\textit{Left:} Sinogram with boundary of $\tA$ having small slope ($c=0.65)$.
The part of the boundary causing the prominent $\xb$-curve of artifacts in the reconstruction
region is highlighted in magenta.  The solid part of the curve indicates the
artifacts that are realized in the reconstruction.  The dotted curve
at the right end of the sinogram indicates potential artifacts that
are not realized because the corresponding part of $\bd(\tA)$ is outside
$\supp(Rf)$ (see Theorem \ref{thm:visible}\ref{thm:ext}).
\textit{Center:} FBP reconstruction.  \textit{Right:} Same FBP reconstruction as in the 
center image highlighting some of the added artifacts. The magenta
curve in the reconstruction is the $\xb$-curve of artifacts 
and the yellow artifact lines are object-dependent artifacts similar
to those in Figure \ref{fig:sqrt large}.}\label{fig:sqrt small}
\end{subfigure} \caption{Illustration of artifacts with smooth
boundary given by \eqref{eq:p_sqrt}.  The $\xb$-curve $\vp\mapsto
\xb\paren{\thbar(\vp)}$ of artifacts is outside the reconstruction
region in the top figure and it meets the object in the bottom
picture.}\label{fig:sqrt}
\end{figure}

According to the condition \eqref{small slope}, the curved part of
$\bd(\tA)$ is the only part that can potentially cause
object-independent artifacts in $D$, since the other parts are
vertical.  In Figure \ref{fig:sqrt}, we consider two data sets $\tA$
with smooth boundary; In Figure \ref{fig:sqrt large}, the
$\xb$-curve $\vp\mapsto \xb\paren{\thbar(\vp)}$ is outside the
unit disk and in Figure \ref{fig:sqrt small}, it meets the object.

 Figure \ref{fig:sqrt large} provides a reconstruction with data set
defined by $c=1.3$ in \eqref{eq:p_sqrt}.  Many artifacts in the
reconstruction region are the same as in Figure \ref{fig:LA-FBP}
because the boundaries of the cutoff regions are substantially the
same: the artifacts corresponding to the circles with $\vp = 4\pi/9$
and the lower circle with $\vp = 5\pi/9$ are the same limited-angle
artifacts as in Figure \ref{fig:LA-FBP} because those parts of the
boundaries are the same.  However, the upper right circled point in
the sinogram has $\vp>5\pi/9$ so the corresponding artifact line has
this larger angle, as seen in the reconstruction.  The center
reconstruction in Figure \ref{fig:sqrt large} shows the $\xb$-curve
of artifacts, but it is far enough from $D$ that it is not visible in the
reconstruction on the right.

Figure \ref{fig:sqrt small} provides a reconstruction with data set
defined by $c=0.65$ in \eqref{eq:p_sqrt}. In this case, the
object-dependent artifacts are similar to those in Figure
\ref{fig:sqrt large}, but the lines for $\vpp$ defined by
\eqref{eq:p_sqrt} are different because $\bd(\tA)$ is different.  The
highlighted part of the boundary of $\tA$ defined by \eqref{eq:p_sqrt}
indicates the boundary points that create the part of the $\xb$-curve
of artifacts that now meets the reconstruction region.  The
highlighted curve in the right-hand reconstruction of Figure
\ref{fig:sqrt small} is this part of the $\xb$-curve.  Note that this
curve is calculated using the formula \eqref{def:xb} for
$\xb\paren{\thbar(\vp)}$ rather than by visually tracing the physical
curve on the reconstruction. That the calculated curve and the
artifact curve are substantially the same shows the efficacy of our
theory.  A simple exercise shows that, for any $c>0$, the $\xb$-curve
changes direction at $\xb(\thbar(1/2+5\pi/9))$.

Let $\vppo$ be the coordinates of the circled point in the upper
right of the sinogram in Figure \ref{fig:sqrt small}.  This circled
point is on the boundary of $\supp(Rf)$ so $\tL\vppo$ is tangent to the
skull and an object-dependent artifact is visible on $\tL\vppo$ in
the reconstruction.  The $\xb$-curve ends at $\xb\paren{\thbar(\vpo)}$
(as justified by Theorem \ref{thm:visible}\,\ref{thm:ext}) and so
the $\xb$-curve seems to blend into this line $\tL\vppo$.  If $\supp(f)$ were
larger and the dotted part of the magenta curve on the sinogram were
in $\supp(Rf)$, the $\xb$-curve would be longer.

\subsection{Region-of-interest (ROI) tomography}\label{sect:ROI} The
ROI problem, also known as interior tomography, is a classical
incomplete data tomography problem in which a part of the object (the
ROI) is imaged using only data over lines that meet the ROI. Such ROI
data are generated, e.g., when the detector width is not large enough
to contain the complete object or when researchers would like a higher
resolution scan of a small part of the object.  In this
section, we apply our theorems to understand ROI CT microlocally,
including the ring artifact at the boundary of the ROI.  We should
point out that practitioners are well aware of the ring artifacts (see
e.g., \cite{EHMCCL-ringROI, CHrB-ringROI}).  Important related work
has been done to analyze the ROI problem (e.g., \cite{FRS, FFRS,
Quinto93, RZ, KR1996, KLM, KR1996}).

First, note that Theorem \ref{thm:visible}\,\ref{thm:all} implies that
all singularities of $f$ in the interior of the ROI are recovered.
This is observed in Figure \ref{fig:ROI}.  If the ROI were not
convex, then all singularities in the interior of its convex hull
would be visible.

 \begin{figure}[h!]%
\centering
 \newcommand{\ww}{0.3\linewidth}
  \includegraphics[height=\ww]{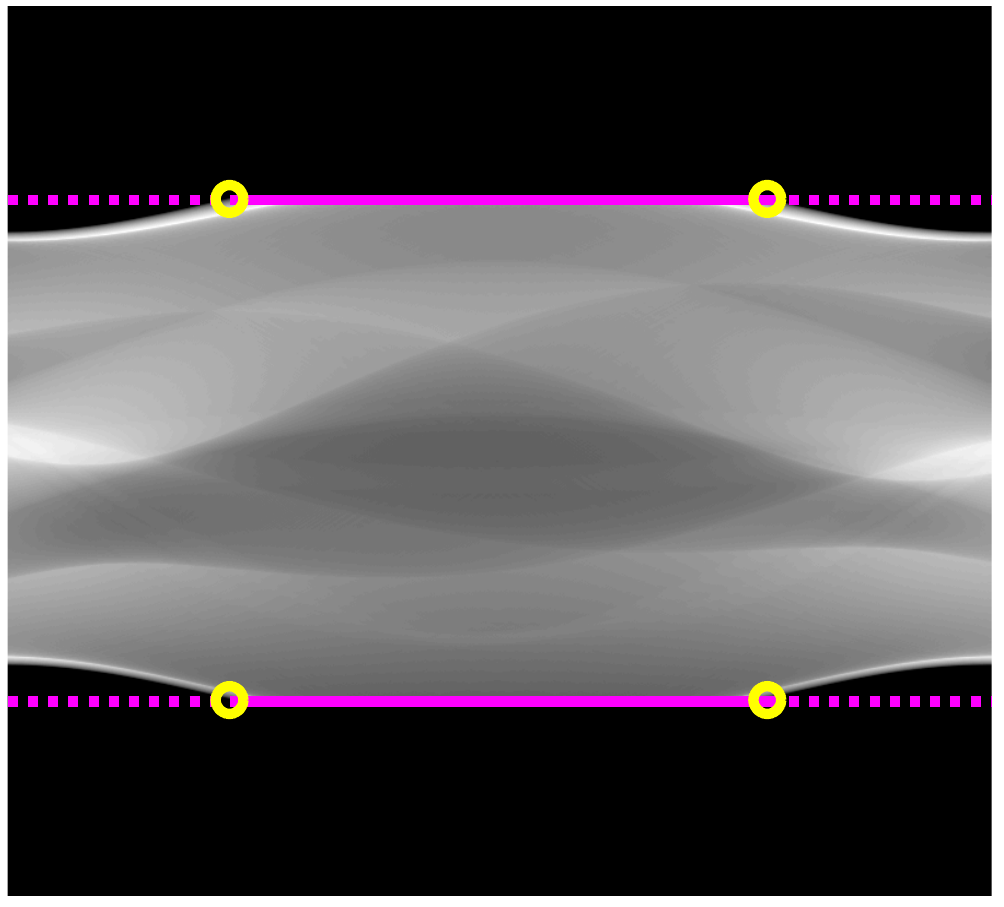}
  \includegraphics[height=\ww]{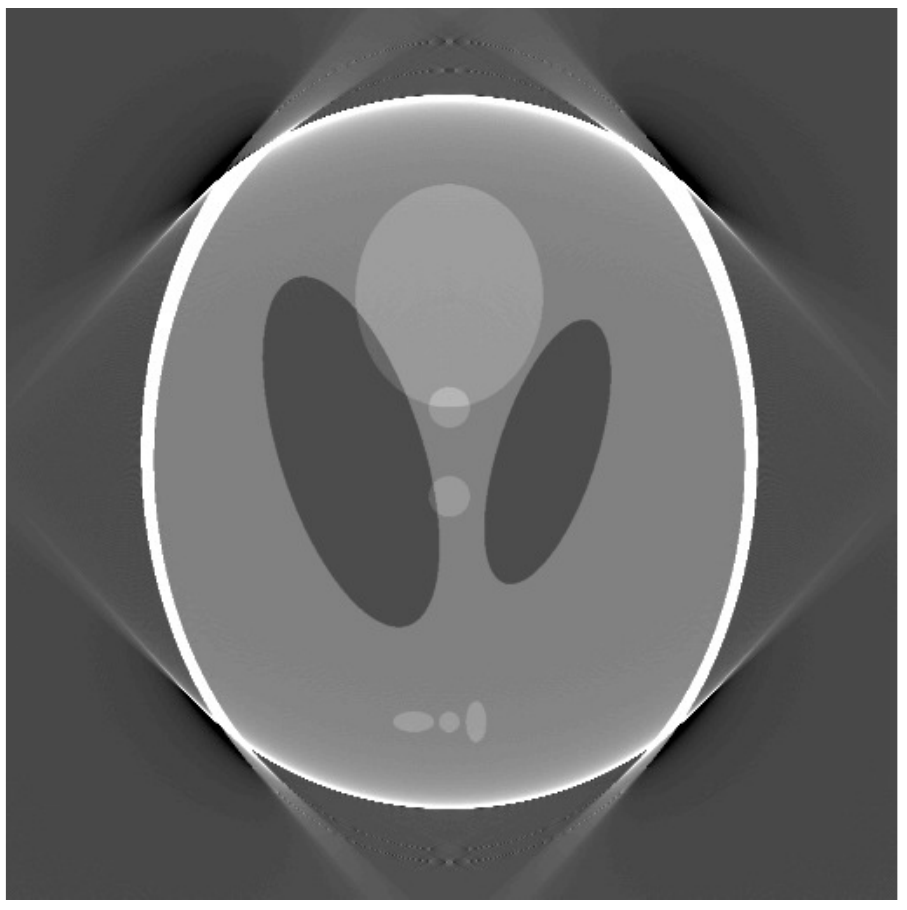}
 \includegraphics[height=\ww]{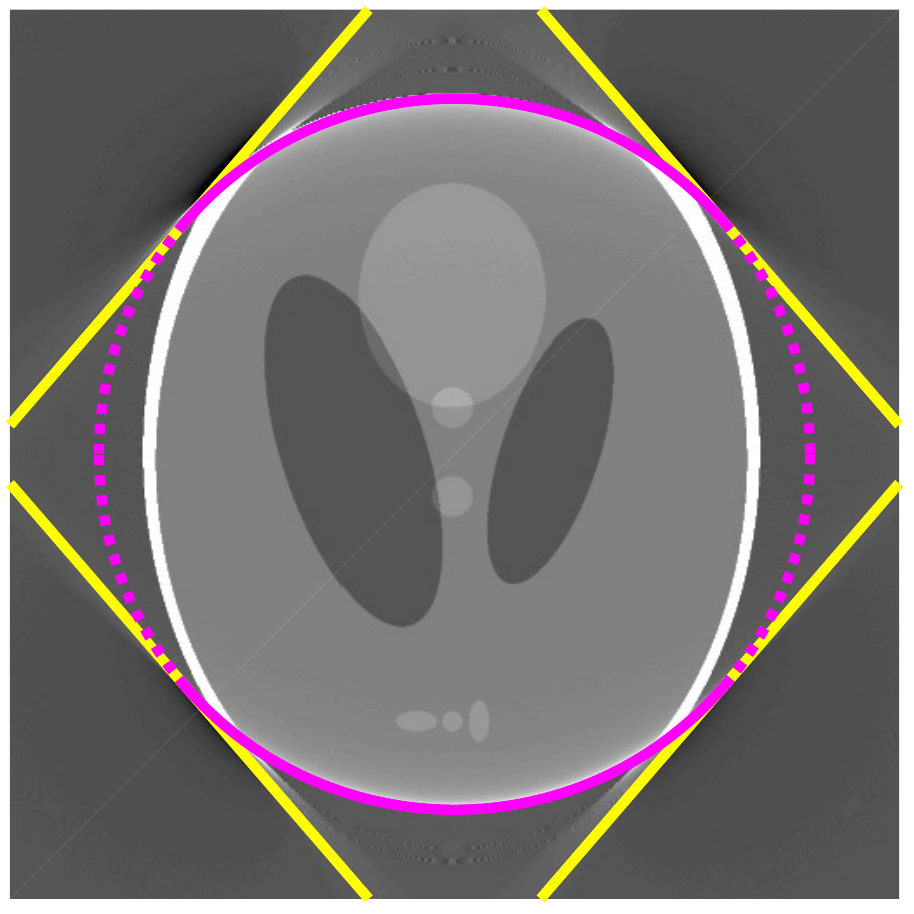}
 \caption{\emph{Left:} ROI data taken within a disk of radius $0.8$
centered at the origin, $p\in [-0.8,0.8]$. The boundary of $\tA$ is
highlighted in magenta.  \emph{Center:} FBP-reconstruction.
\emph{Right} Same FBP reconstruction as in the center image,
highlighting the $x_{b}$-curve of artifacts in magenta and the
object-dependent streak artifacts in yellow.}\label{fig:ROI} 
\end{figure}

The boundary of the sinogram in Figure~\ref{fig:ROI} is given by
horizontal lines $p=\pm 0.8$.  Since $p'=0$, the $\xb$-curve
\eqref{def:xb} is given by $\xb\paren{\thbar(\vp)}=0.8\cdot
\thbar(\vp)$, which is a circle of radius $0.8$.  The $\xb$-artifact-circle 
is highlighted in the right reconstruction of Figure
\ref{fig:ROI}, but it can be also be seen clearly in the top and
bottom of the center reconstruction, even without the highlighting.
However, the artifact circle does not extend outside the object (as
represented by the dotted magenta curve in the reconstruction and
which comes from the dotted segments of $\bd(\tA)$ in the sinogram)
because $Rf$ is zero near the corresponding lines.  Theorem
\ref{thm:curve}\,\ref{no sing}\ref{Rf=0} can be used to explain the
invisible curve.

One also sees object-dependent artifacts described by Theorem
\ref{thm:streak}\,\ref{thmPart:object dep} in
Figure~\ref{fig:ROI}.  For example, streak artifacts occur on 
the lines $\tL\vppo$ corresponding to the four circled points
$\vppo$ in $\bd(\tA)$ in the sinogram.
These lines $\tL\vppo$ are tangent to the outer boundary of the
skull, therefore $f$ has wavefront set directions normal to these
lines, and this causes the artifacts by Theorem
\ref{thm:streak}\,\ref{thmPart:object dep}

In general, one can show that if the ROI is strictly convex with
smooth boundary then the $\xb$-curve of artifacts traces the boundary
of the ROI. The proof is an exercise using the parametrization in
$\vpp$ of tangent lines to this boundary.

\subsection{The general case}\label{sect:general} The reconstruction
in Figure \ref{fig:general artifacts} illustrates all of our cases in
one.  In that figure, we consider a general incomplete data set with a
rectangular region cut out of the sinogram leading to all considered
types of artifacts. Now, we describe the resulting artifacts.  In
Figure~\ref{fig:general artifacts} the horizontal sinogram boundaries
at $p = \po =\pm 0.35$ for $\phi \in \bparen{\frac{7}{18}\pi,
\frac{11}{18}\pi}$ are displayed in solid magenta line. As in the ROI
case, on these boundaries, we have $p' = 0$ and thus circular arcs
of radius $\po$ for the given interval for $\vp$ are added in the
reconstruction (as indicated by solid magenta).  As predicted by
Theorem \ref{thm:streak}\,\ref{thmPart:nonsmooth}, each of the four
corners produce a line artifact as marked by the yellow solid lines in
the right-hand reconstruction, and they align tangentially with the
ends of the curved artifacts.
\begin{figure}[h!]
\centering
\newcommand{\ww}{0.3\linewidth}%
\includegraphics[height=\ww]{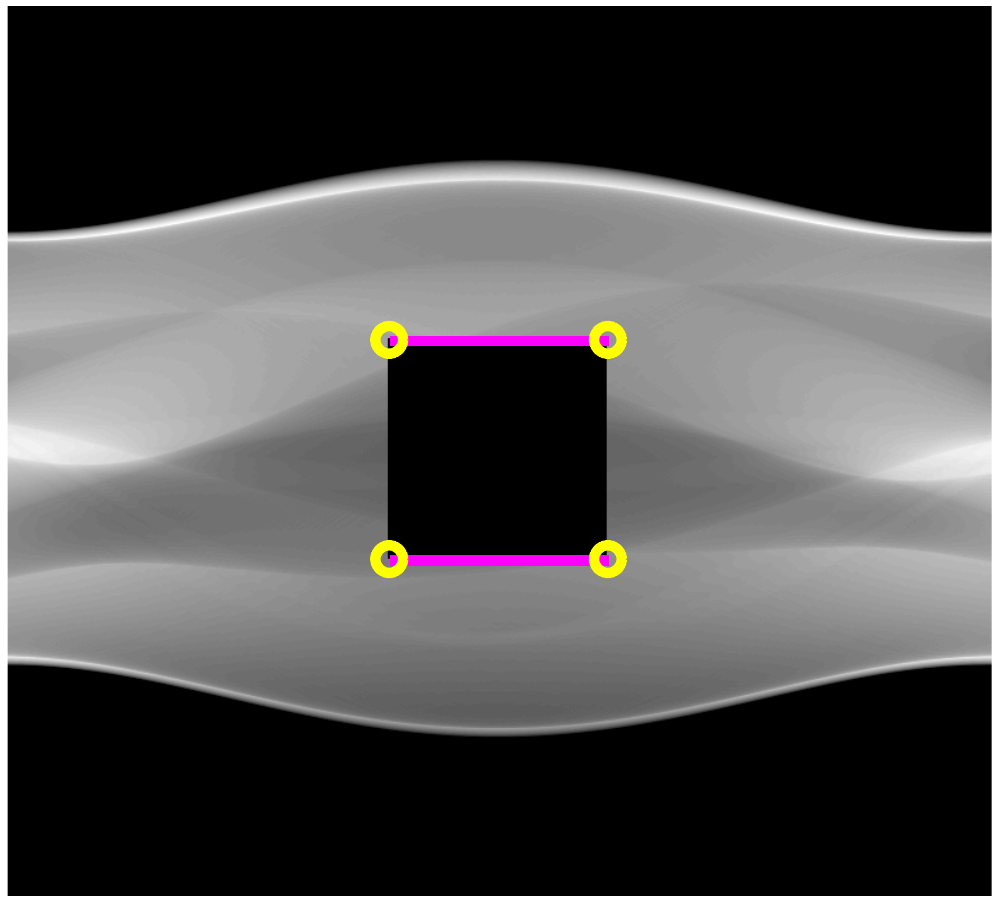}
\includegraphics[height=\ww]{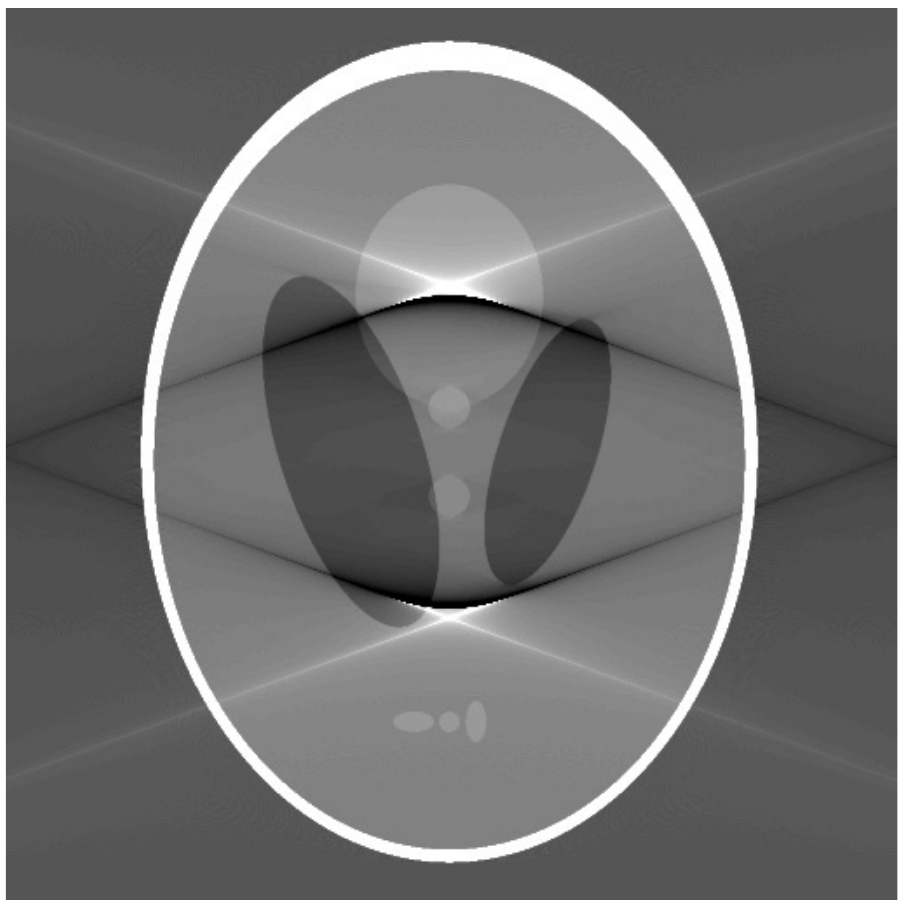}
\includegraphics[height=\ww]{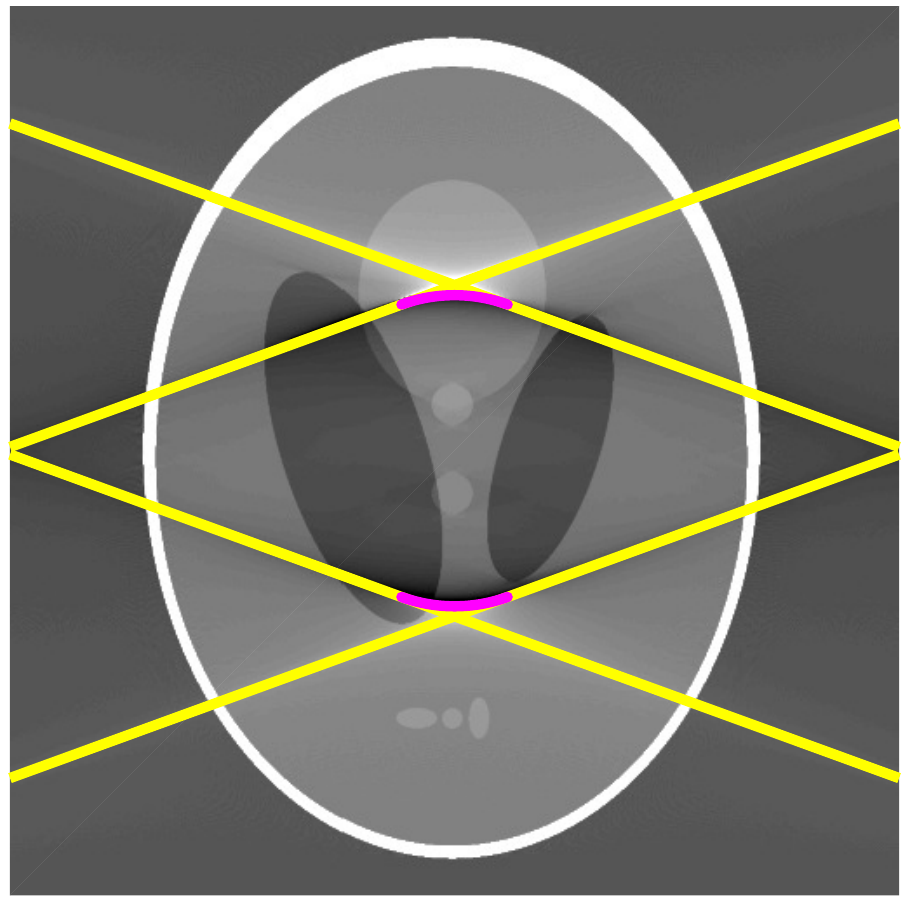}
\caption{ \textit{Left:} The sinogram for a general incomplete data
problem in which the cutoff region, $\tA$, has a locally smooth boundary
with zero and infinite slope as well as corners.  The cutout from the
sinogram is at $\frac{7\pi}{18}$ and $\frac{11\pi}{18}$, $p=\pm0.35$.
\textit{Center:} FBP reconstruction.  \textit{Right:} Same reconstruction
with the circular $\xb$-curve of artifacts highlighted in magenta and
object-independent ``corner'' streak artifacts highlighted in yellow.}
\label{fig:general artifacts}
\end{figure}

The circular arc between those lines corresponds to the top and
bottom parts of $\bd(\tA)$ as the data are, locally, constrained as in
ROI CT (see Section  \ref{sect:ROI}).

In Figure~\ref{fig:general artifacts}, there are other
object-dependent streaks corresponding to the vertical lines in the
sinogram at $\vp = \frac{7\pi}{18}$ and at $\vp=\frac{11\pi}{18}$ as
predicted by Theorem \ref{thm:streak}\,\ref{thmPart:object dep},
but they are less pronounced and more difficult to see.

\subsection{Summary}\label{sect:summary} 

We have presented reconstructions that illustrate all of types of
incomplete data and each of our theorems from Section
\ref{sect:results}.  All artifacts arise because of points $\vppo\in
\bd(\tA)$, and they fall into two categories.
\begin{itemize}
\item Streak artifacts on the line $\tL\vppo$:
\begin{itemize}
\item Object-dependent streaks occur when $\bd(\tA)$ is smooth at
$\vppo$ and a singularity of $f$ is normal to $\tL\vppo$.  \item
Object-independent streaks occur when $\bd(\tA)$ is nonsmooth at
$\vppo$.
\end{itemize}

\item Artifacts on curves are always object-independent, and they are
generated by the map $\vp\mapsto \xb\paren{\thbar(\vp)}$ from parts of
$\bd(\tA)$ that are smooth and of small slope. 
\end{itemize}

\section{Strength of added artifacts}\label{sect:strength}

In this section, we go back to parametrizing lines by $\thp\in \Sor$.

Using the Sobolev continuity of $Rf$, one can measure the strength in
Sobolev scale of added artifacts in several useful cases.  First, we
define the Sobolev norm \cite{Rudin:FA, Petersen}.  We state it for
distributions, therefore, it will apply to functions $f\in \Lloc(D)$.

 \begin{definition}[Sobolev wavefront set
\cite{Petersen}]\label{def:Hs norm} For $s\in \rr$, the Sobolev space
   $H_s(\rn)$ is the set of all distributions with locally
   square-integrable Fourier transform and with finite Sobolev norm:
   \bel{Hs norm} \norm{f}_s: = \paren{\int_{y\in \rn} \abs{\Fc f(y)}^2
     (1+\norm{y}^2)^s\,dy}^{1/2}<\infty.\ee Let $f$ be a distribution and
   let $\xo\in \rn$ and $\xio\in \rn\smo$.  We say $f$ is in $H^s$ at $\xo$
   in direction $\xio$ if there is a cutoff function $\psi$ at $\xo$ and an
   open cone $V$ containing $\xio$ such that the \emph{localized and
     microlocalized Sobolev seminorm} is finite: \bel{microlocal Hs
     seminorm} \norm{f}_{s,\psi,V}:= \paren{\int_{y\in V}
     \abs{\Fc\paren{\psi f}(y)}^2 (1+\norm{y}^2)^s\,dy}^{1/2}<\infty.\ee

If \eqref{microlocal Hs seminorm} does not hold for any cutoff
function at $\xo$, $\psi$, or any conic neighborhood $V$ of $\xio$,
then we say that $(\xo,\xio)$ is in the Sobolev wavefront set of $f$
of order $s$, $(\xo,\xio)\in \WF_s(f)$.
\end{definition}

An exercise using the definitions
shows that $\WF(f) = \cup_{s\in\rr} \WF_s(f)$ (see
\cite{Friedlander98}).

The Sobolev wavefront set can be defined for measurable functions $g$ on
$\Sor$ using the identification \eqref{def:tg} that reduces to this
definition for $\tg(\vp,p)=g\paren{\thbar(\vp),p)}$.

Note that this norm on distributions on $\Sor$ is not the typical
$H_{0,s}$ norm used in elementary continuity proofs for the Radon
transform (see e.g., \cite[equation (2.11)]{HahnQ}), but this is the
appropriate norm for the continuity theorems for general Fourier
integral operators \cite[Theorem 4.3.1]{Ho1971}, \cite[Corollary
4.4.5]{Duistermaat}.

Our next theorem gives the strength in Sobolev scale of added
singularities of $\LA f$ under certain assumptions on $f$. It uses the
relation between microlocal Sobolev strength of $f$ and of $Rf$,
\cite[Theorem 3.1]{Quinto93} and of $g$ and $R^* g$, which is given in
Proposition \ref{prop:Sobolev R*} (see also \cite{KLM} for related
results).

\begin{theorem}\label{thm:Sobolev sing}  Let $f\in \LD$ and let
$A\subset \Sor$ satisfy Assumption \ref{hyp:A}.  Let $\thpo\in \bd(A)$ and
assume $Rf\thpo\neq 0$ and $f$ is smooth normal to $L\thpo$, i.e.,
$\WFLo(f)=\emptyset$.
\begin{enumerate}[label=\Alph*.]
\item \label{smooth bndy not vertical} Assume $\bd(A)$ is smooth and
not vertical at $\thpo$. Let $\xb = \xb(\tho)$ be given by
\eqref{def:xb} and let $\om\neq 0$.  Then, $\LA f$ is in $H_s$ for
$s<0$ at $\xio=(\xb,\om\thbar(\tho))$ and $\xio\in \WF_{0}(\LA f)$.
Thus, there are singularities above $\xb$ in the $0$-order wavefront
set of $\LA f$.

\item \label{nonsmooth bndy} Now, assume $\bd(A)$ has a corner at
$\thpo$ (see Definition \ref{def:smooth}).  Then for each $(x,\xi)\in
\NLo$, $(x,\xi)\in \WF_{1}(\LA f)$ and, except for two points on
$L\thpo$, $\LA f$ is in $H_s$ for $s<1$ at $(x,\xi)$.  If one of the
two one-sided tangent lines to the corner is vertical, then there is
only one such point.
\end{enumerate}
\end{theorem}

This theorem provides estimates on smoothness for more general
data sets than the limited-angle case, which was thoroughly considered
in \cite{Ka1997:limited-angle,Nguyen2015ip}. In contrast to part
\ref{smooth bndy not vertical} of this theorem, if $\bd(A)$ has a
vertical tangent at $\thpo$, then, under the smoothness assumption on
$f$, there are no added artifacts in $\LA f$ normal to $L\thpo$ (see
Theorem \ref{thm:streak}\,\ref{thmPart:object dep}).  Part \ref{smooth
bndy not vertical} of this theorem is a more precise version of
Theorem \ref{thm:curve}\,\ref{Rf neq 0}.  Under the assumptions in
parts \ref{smooth bndy not vertical} and \ref{nonsmooth bndy}, $\bd(A)$
\emph{will cause} specific singularities in specific locations on
$L\thpo$. The two more singular points in part \ref{nonsmooth bndy}
are specified in equation \eqref{two xbs}.  If one part of $\bd(A)$ is
vertical at $\thpo$, then there is only one such more singular point.

  This theorem will be proven in Section \ref{appendix:Sobolev Sing}
of the appendix.

\section{Artifact reduction}\label{sect:reduction}

In this section, we briefly describe a method to suppress the added
streak artifacts described in Theorems \ref{thm:curve} and
\ref{thm:streak}.  This is a standard technique for many
practitioners, but it is worth highlighting because it is simple and
useful. 

As outlined in Section \ref{sect:results}, the application of FBP to
incomplete data extends the data from $A\subset \Sor$ to all of $\Sor$ by
padding it with zeros on the complement of $A$. This hard truncation can
create discontinuities on $\bd(A)$ and that explains the artifacts.  These
jumps are stronger singularities than those of $Rf$ for $Rf\in
H_{1/2}(\Sor)$ since $f\in \LD = H_0(D)$.

One natural way to get rid of the jump discontinuities of
$\onea$ is to replace $\onea$ by a smooth function on $\Sor$, $\psi$,
that is equal to zero off of $A$ and equal to one on most of
$\intt(A)$ and smoothly transitions to zero near $\bd(A)$.  We also
assume $\psi$ is symmetric in the sense $\psi(\th,p) = \psi(-\th,-p)$
for all $\thp$.  This gives the forward operator \bel{def:Rp}\Rp f
\thp = \psi\thp Rf\thp\ee and the reconstruction operator
\bel{def:Lp}\Lp f=R^*\paren{\Lambda \Rp f}=R^*\paren{\Lambda \psi R
f}.\ee Because $\psi$ is a smooth function, $\Rp$ is a standard
Fourier integral operator and so $\Lp$ is a standard
pseudodifferential operator.  This allows us to show that $\Lp$ does
not add artifacts.

\begin{theorem}[Artifact Reduction Theorem]\label{thm:reduction}
Let $f\in \LD$ and let $A\subset \Sor$ satisfy Assumption \ref{hyp:A}.
Then \bel{regularity} \WF(\Lp f)\subset \WF(f).\ee Therefore, $\Lp$
does not add artifacts to the reconstruction.

Let $x\in D$, $\th\in \So$, and $\om\neq 0$. If
$\psi(\th,x\cdot\th)\neq 0$, then \bel{ellipticity}
(x,\om\th)\in \WF(\Lp f)\ \text{ if and only if }\
(x,\om\th)\in \WF(f).\ee
\end{theorem}

Theorem \ref{thm:reduction} is a special case of a known result in e.g.,
\cite{KLM} or the symbol calculation in \cite{Quinto1980} and is
stated for completeness.  This theorem shows the advantages of
including a smooth cutoff, and it has been suggested in several
settings, including limited-angle X-ray CT \cite{Ka1997:limited-angle,
FrikelQuinto2013} and more general tomography problems
\cite{FrikelQuinto-hyperplane, FrikelQuinto2015, KLM,SQWJ-ET}.  More
sophisticated methods are discussed in \cite{BJS-techreport, Borg2017}
for the synchrotron problem that is described in Section
\ref{sect:synchrotron}.

Although this artifact reduction technique does not create any
\emph{singular} artifacts in $\Lp f$, it can turn \emph{singular}
artifacts into \emph{smooth} artifacts, for example, by smoothing
$\xb$-curves.

Figure \ref{fig:general smoothed} illustrates the efficacy of this
smoothing algorithm on simulated data, and Figure \ref{fig:improved
synchrotron reconstruction} in Section \ref{sect:synchrotron}
demonstrates its benefits on real synchrotron data. 

\begin{figure}[t]
 \newcommand{\ww}{0.3\linewidth}
\centering
 \includegraphics[height=\ww]{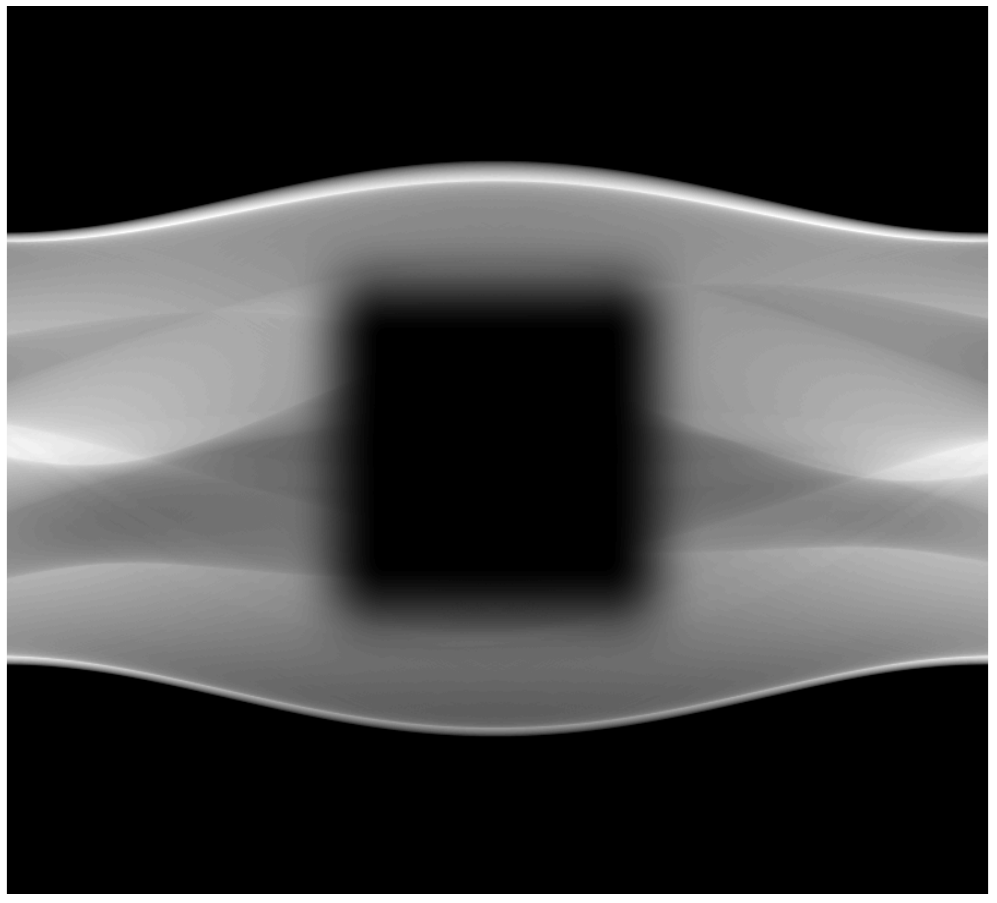}
\includegraphics[height=\ww]{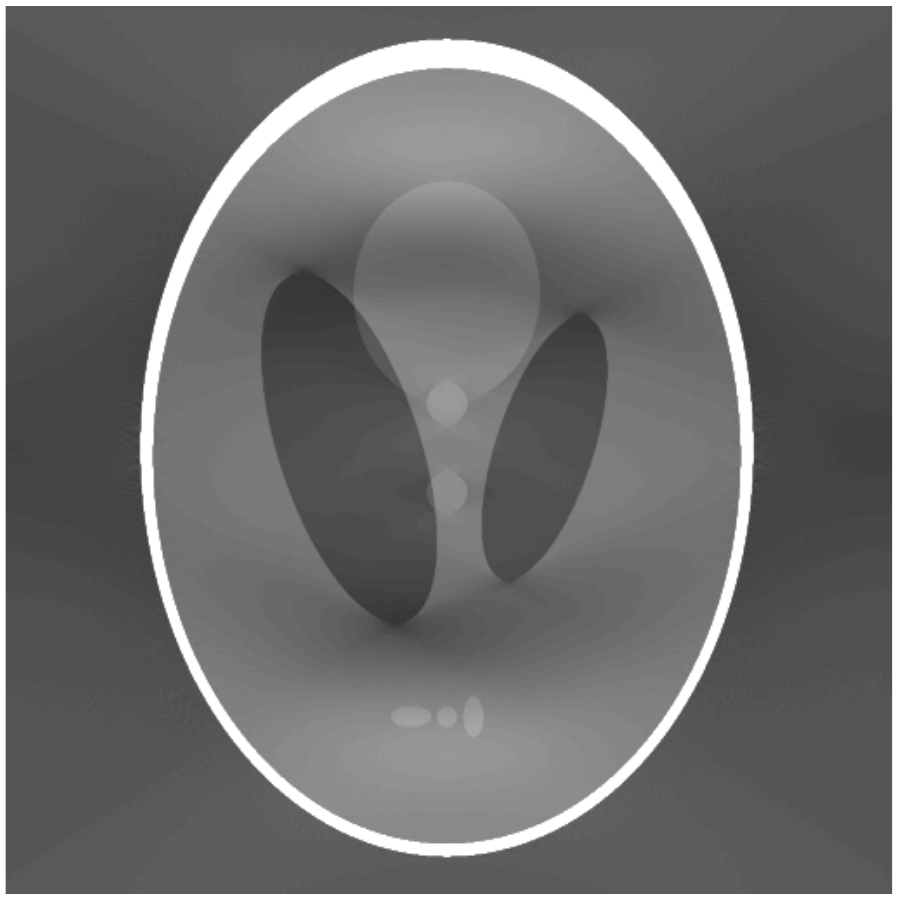}
\includegraphics[height=\ww]{SINOBOXrecon_lowcontrast}
\caption{ \emph{Left:} Smoothed sinogram.  \emph{Center:} Smoothed
reconstruction with suppressed artifacts. \emph{Right:} Reconstruction
using $\LA$, with sharp cutoff.} \label{fig:general smoothed}
\end{figure}

	\section{Application: a synchrotron 
	experiment}\label{sect:synchrotron}

	In this section, we use the identifications given in
	\eqref{otpr-conventions} and show sinograms as subsets of the $\vpp$
	plane.

	\begin{figure}[h!]
	 \newcommand{\ww}{0.3\linewidth} \centering
\includegraphics[width=\ww,height=\ww]{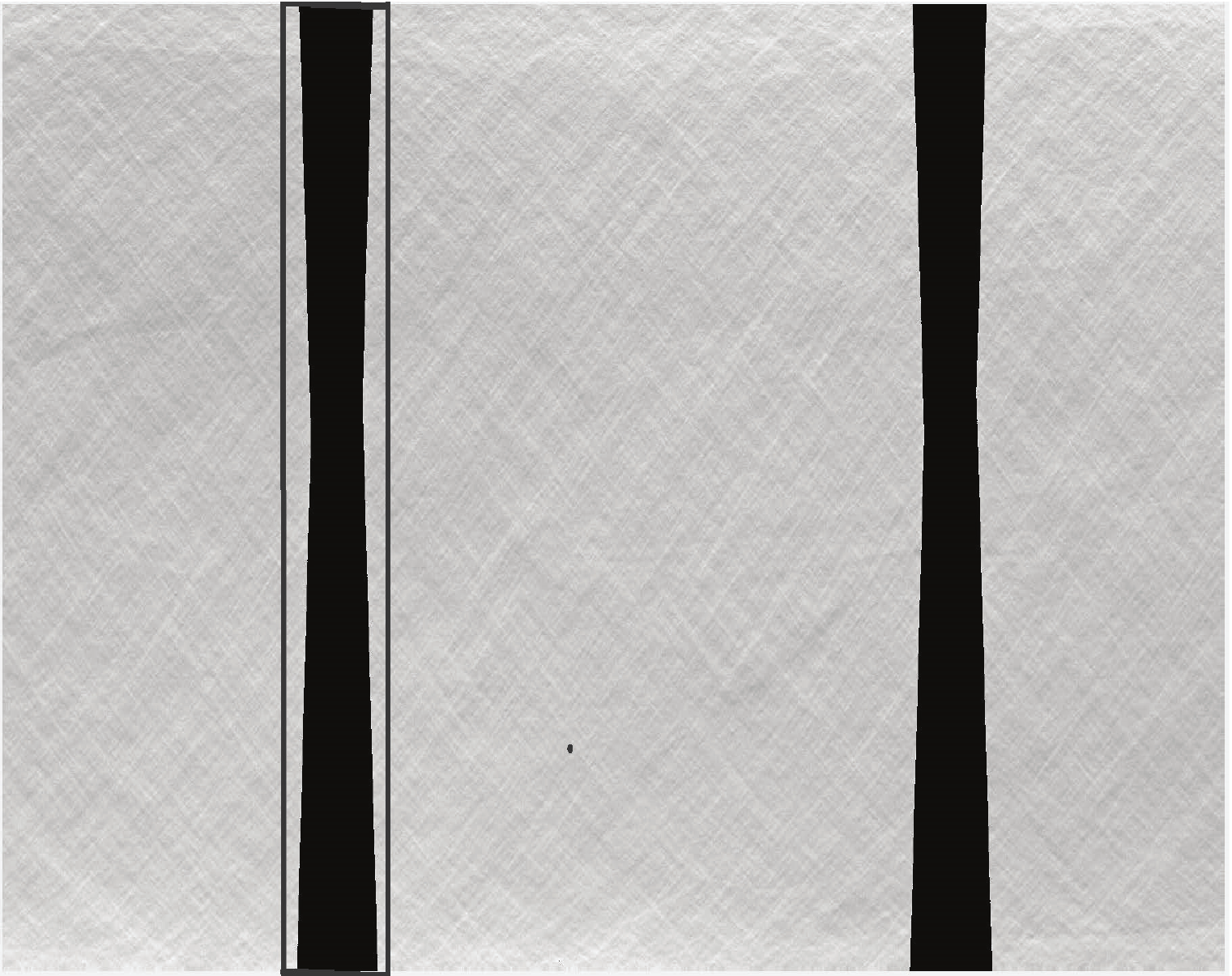}
\includegraphics[width=\ww,height=\ww]{SinSegmentedZoom}
\includegraphics[width=\ww]{IOP-Fig6-zoom-1}
	\caption{\textit{Left:} The truncated attenuation sinogram (after
	processing to get Radon transform data).  \textit{Center:} the
	enlargement of the section of $\bd(\tA)$ between the two dark vertical
	lines in the left-hand sinogram.  \textit{Right:} Zoom of the
	corresponding reconstruction.  \cite[\copyright IOP Publishing.
	Reproduced by permission of IOP Publishing.  All rights
	reserved]{Borg2017}.}\label{fig:sinogram-streaks}
	\end{figure}

	Figure \ref{fig:sinogram-streaks} shows tomographic data of a
	chalk sample (sinogram on the left and a zoomed version in the center)
	that was acquired by a synchrotron experiment
	\cite{BJS-techreport,Borg2017} (see \cite{LQP} for related work). In
	the right picture of Figure \ref{fig:sinogram-streaks} a zoom of the
	corresponding reconstruction is shown (see also Figure
	\ref{subfig:sync data fbp rec}). As can be clearly observed, the
	reconstruction includes dramatic streaks that are independent of the
	object.  These streaks motivated the research in this article since
	they were not explained by the mathematical theory at that time (such
	as in \cite{Ka1997:limited-angle, Nguyen2015ip, FrikelQuinto2013,
	FrikelQuinto2015, FrikelQuinto-hyperplane}).

	\begin{figure}[h!]
	\centering\includegraphics[height=4cm]{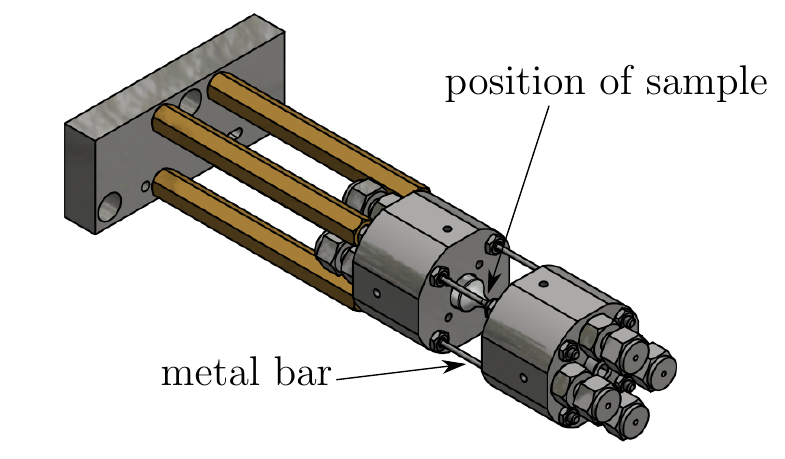}
	\includegraphics[height=4cm]{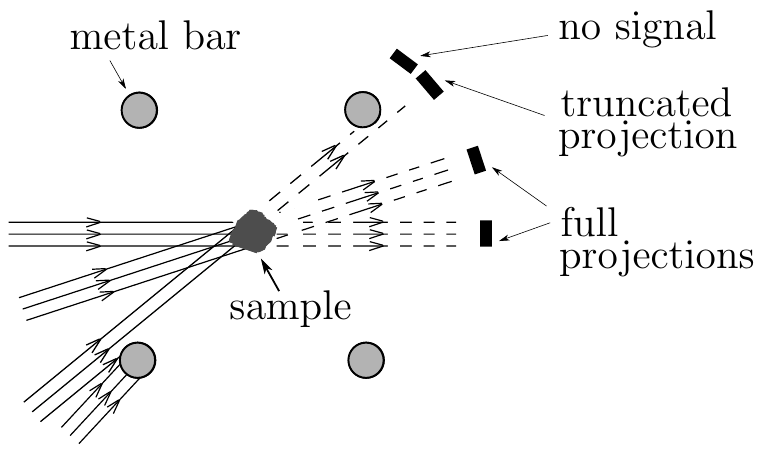}

		\caption{Data acquisition setup for the synchrotron experiment
		  \cite[\copyright IOP Publishing. Reproduced by permission of IOP
		  Publishing.  All rights reserved]{Borg2017}.}
		\label{fig:synchrotron setup}
	\end{figure}

	Taking a closer look at the attenuation sinogram and its zoom in Figure 
	\ref{fig:sinogram-streaks} a staircasing is revealed with
	vertical and horizontal boundaries.  This is a result of X-rays being
	blocked by four metal bars that help stabilize the percolation chamber
	(sample holder) as the sample is subjected to high pressure during
	data acquisition, see Figure \ref{fig:synchrotron setup}.  More
	details are given in \cite{Borg2017}.

	Because the original reconstructions of this synchrotron data used a
	sharp cutoff, $\onea$, the reconstructions suffer from severe
	streak artifacts as can be seen in Figure \ref{subfig:sync data fbp
	rec}.  These artifacts are exactly described by Theorem
	\ref{thm:streak}\,\ref{thmPart:nonsmooth} in that each corner of the
	sinogram gives rise to a line artifact in the reconstruction (cf.
	left and center image in Figure \ref{fig:sinogram-streaks}).  The
	authors of \cite{Borg2017} then use a smooth cutoff function at
	$\bd(\tA)$ that essentially eliminates the streaks.  The resulting
	reconstruction is shown in Figure \ref{subfig:sync data fbp rec ours}
	below.

	\begin{figure}[h!]
	\centering
	\newcommand{\www}{0.455\textwidth}
	\newcommand{\ww}{0.225\textwidth}
	\begin{subfigure}[t]{0.8\textwidth}
	\begin{multicols}{2}\begin{center}
\includegraphics[width=\www, height = \www]{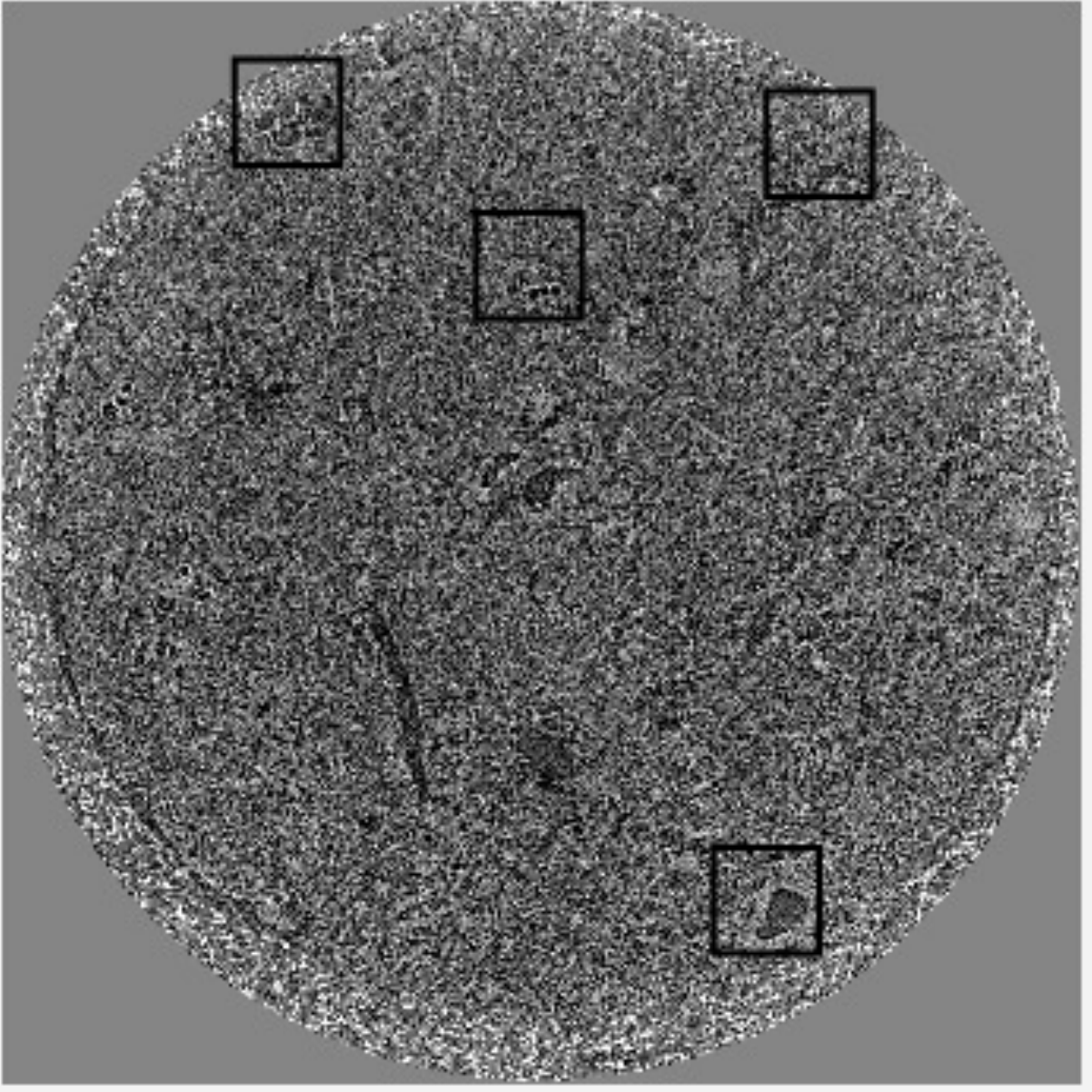}\end{center}
	\newpage\begin{center}
\includegraphics[width=\ww]{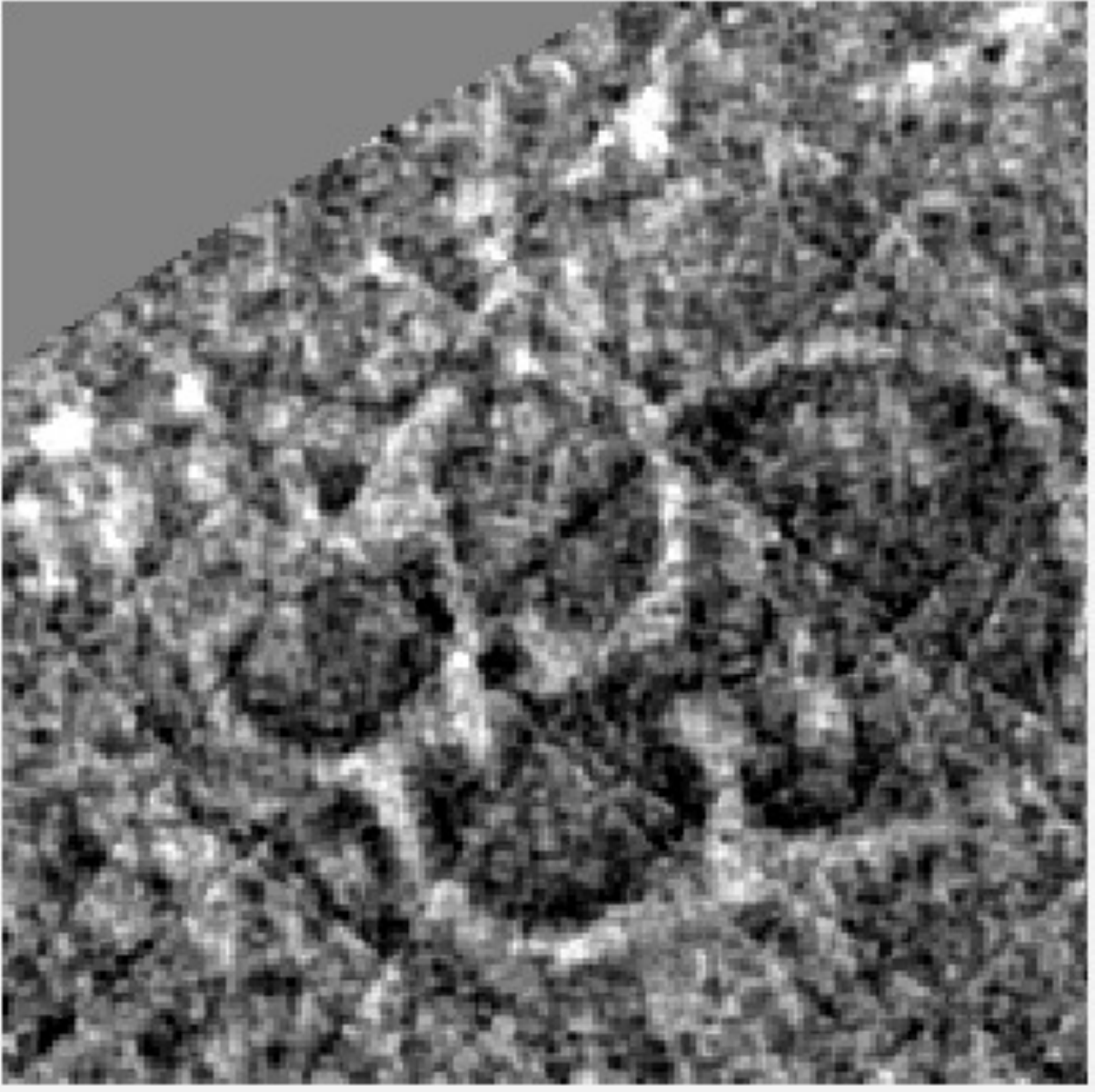}
\includegraphics[width=\ww]{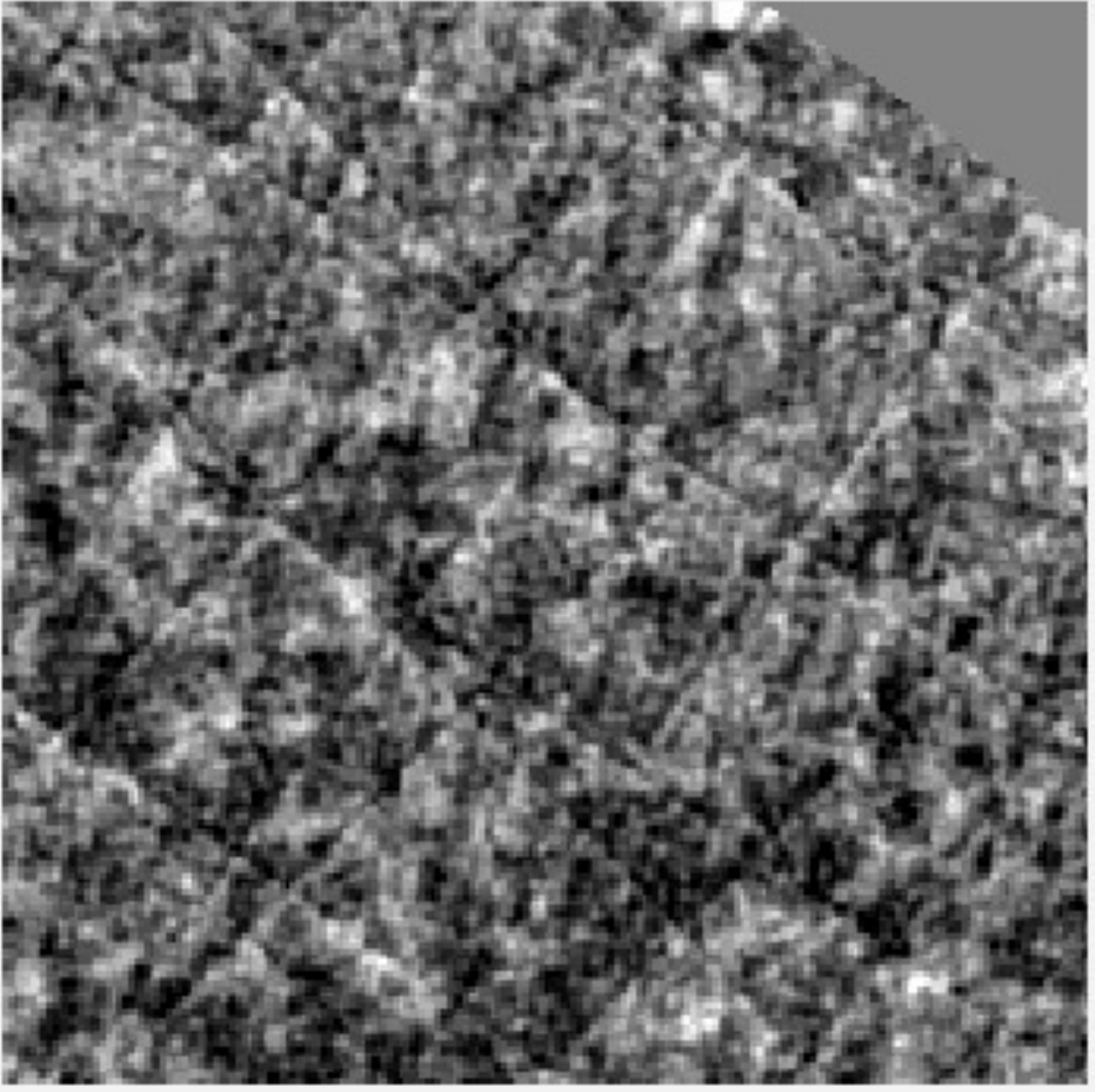}
\includegraphics[width=\ww]{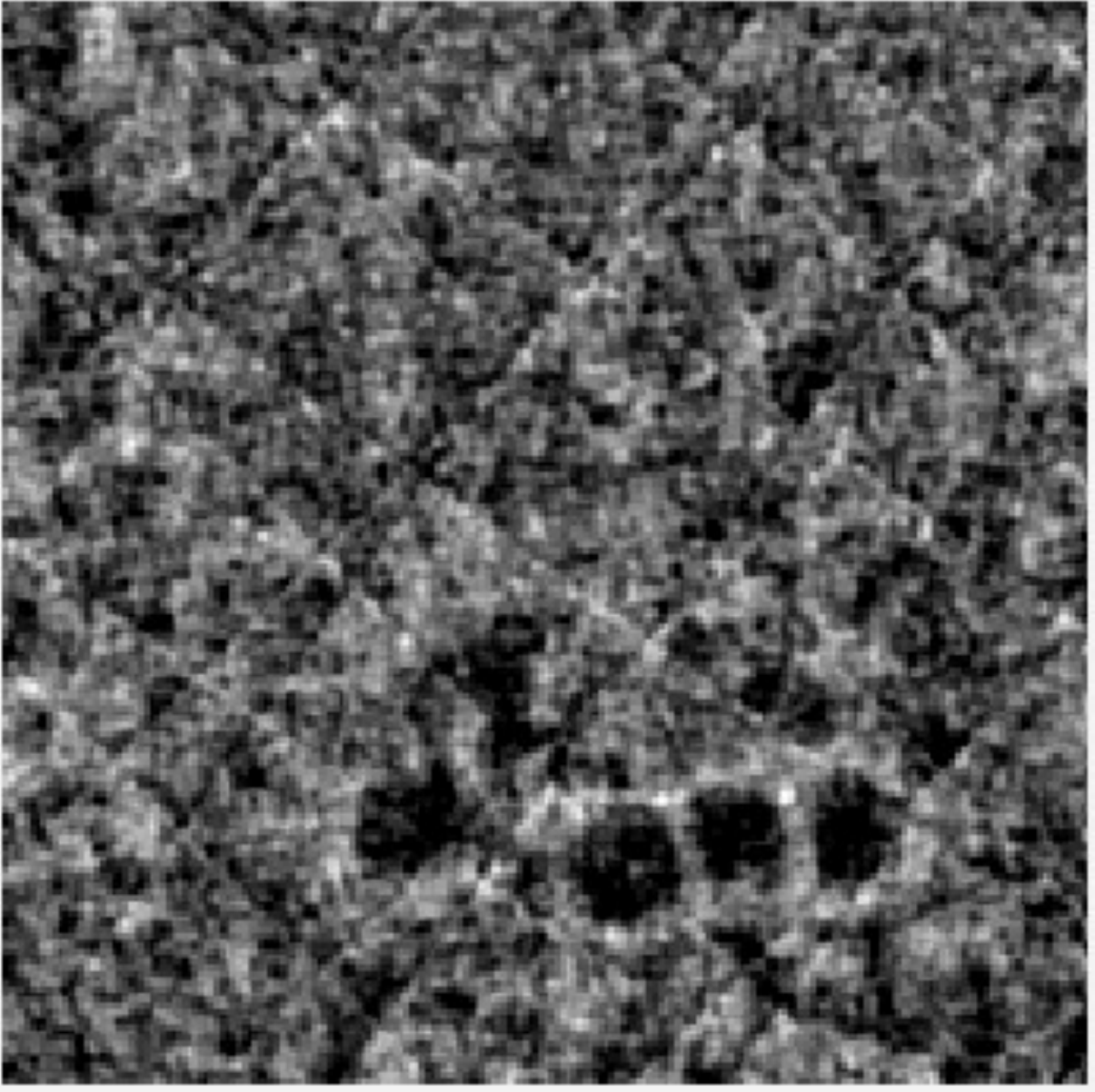}
\includegraphics[width=\ww]{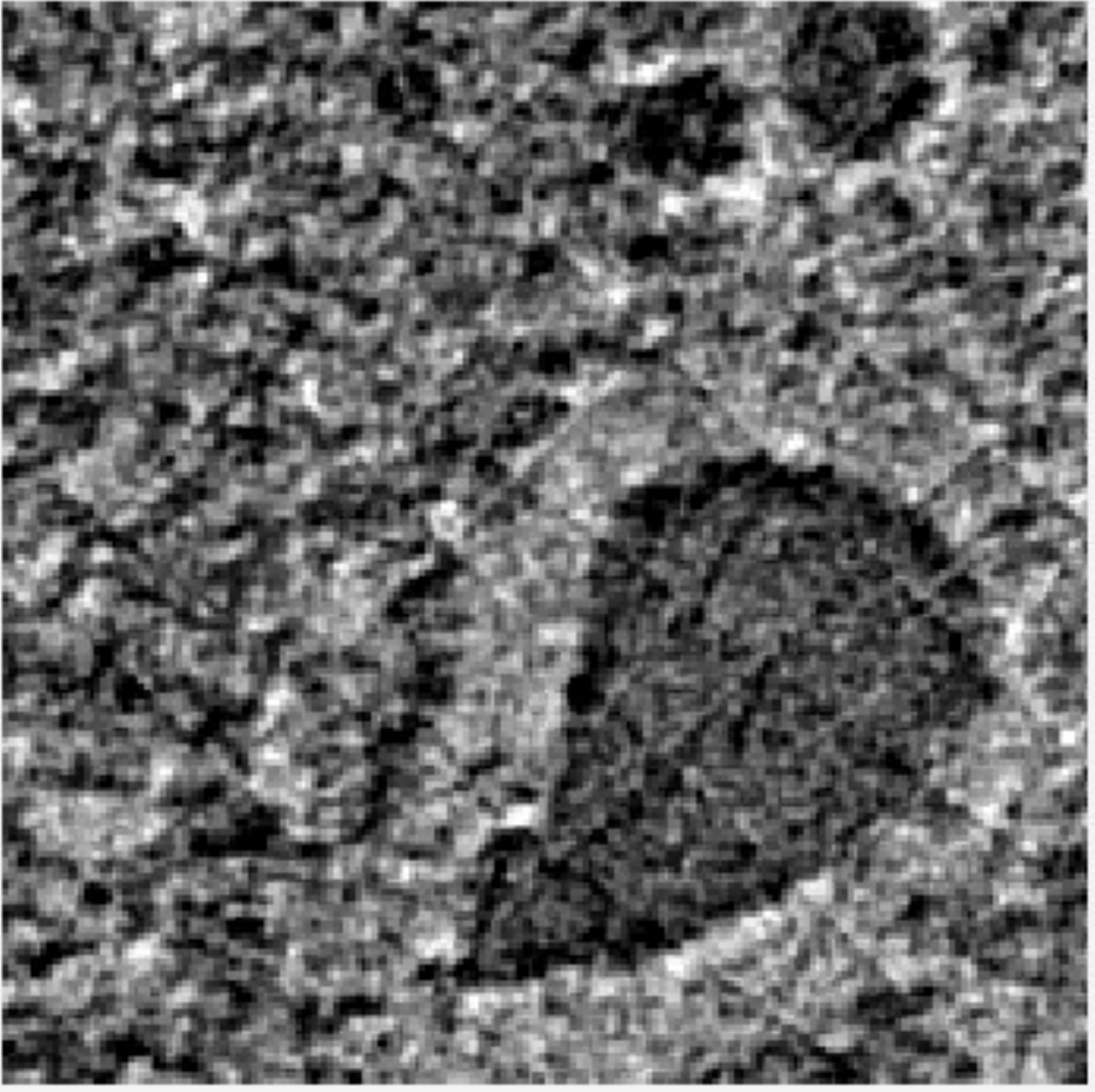}
	\end{center}
	\end{multicols}
		\caption{Standard FBP reconstruction}
		\label{subfig:sync data fbp rec}
	\end{subfigure}
	\hskip2ex
	\begin{subfigure}[t]{0.8\textwidth}
	\begin{multicols}{2}
	\begin{center}\includegraphics[width=\www, 
	height = \www]{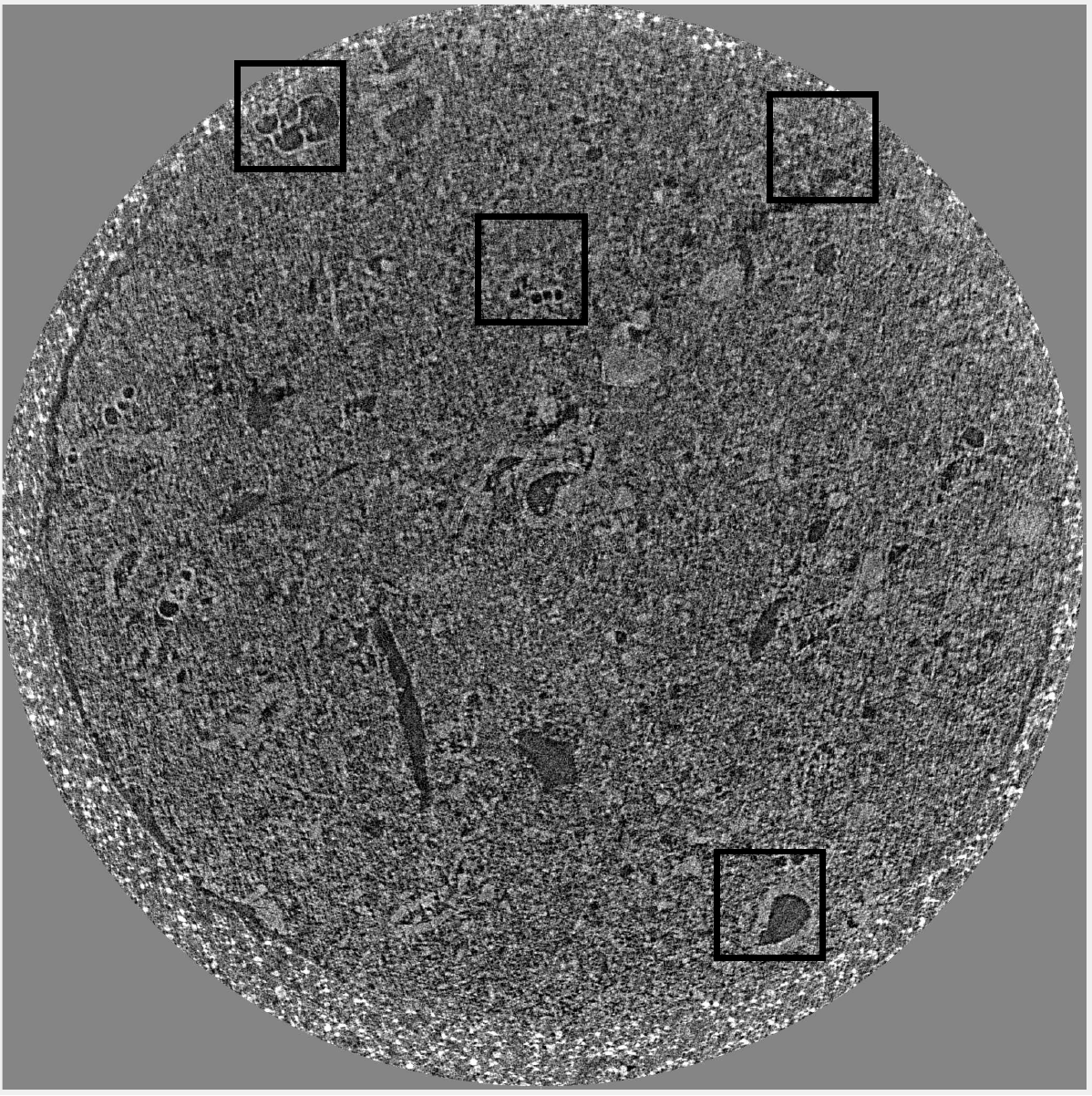}\end{center}
	\newpage
	\begin{center}
	\includegraphics[width=\ww]{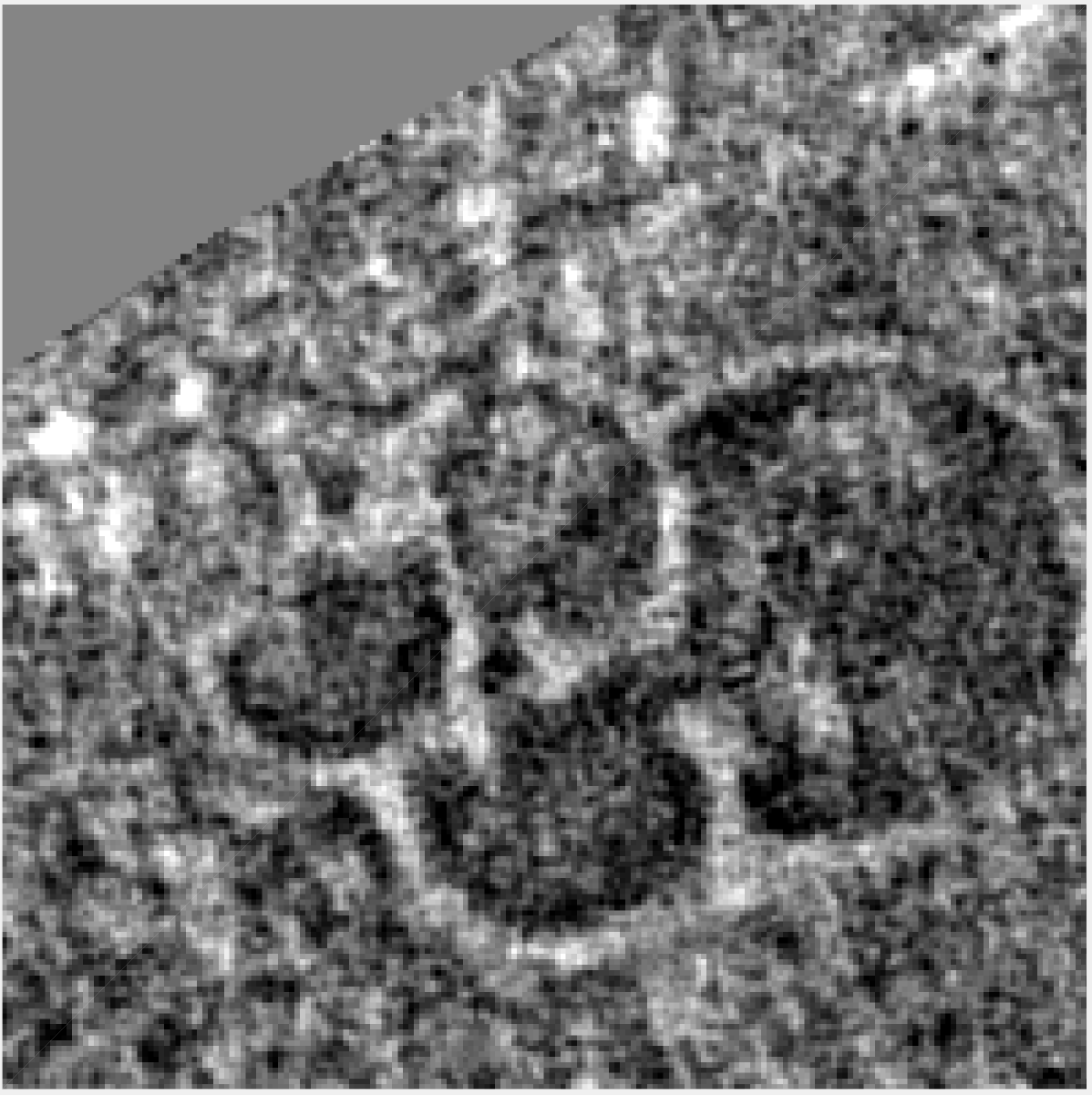}
	\includegraphics[width=\ww]{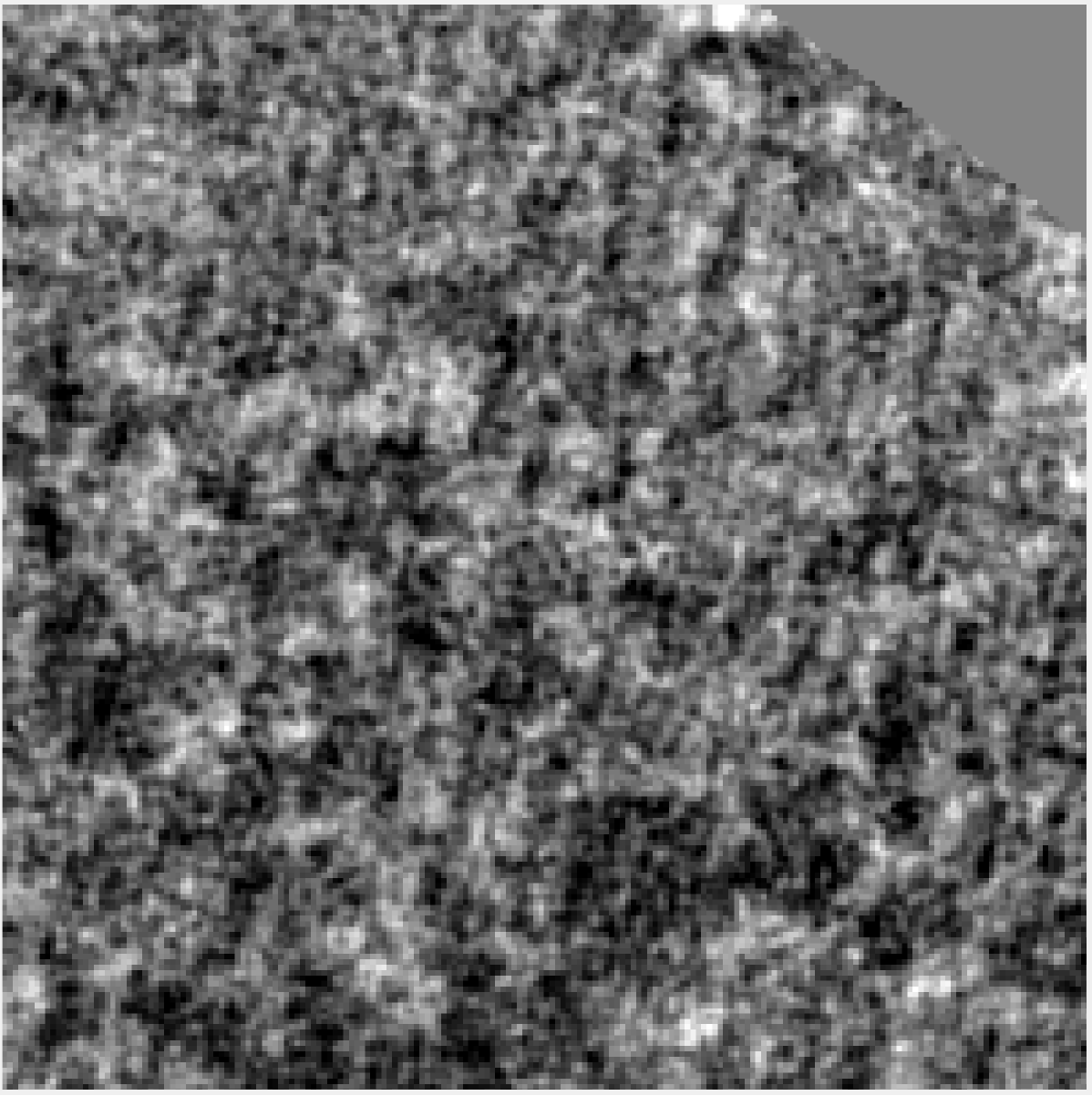}
	\vspace{1mm}\includegraphics[width=\ww]{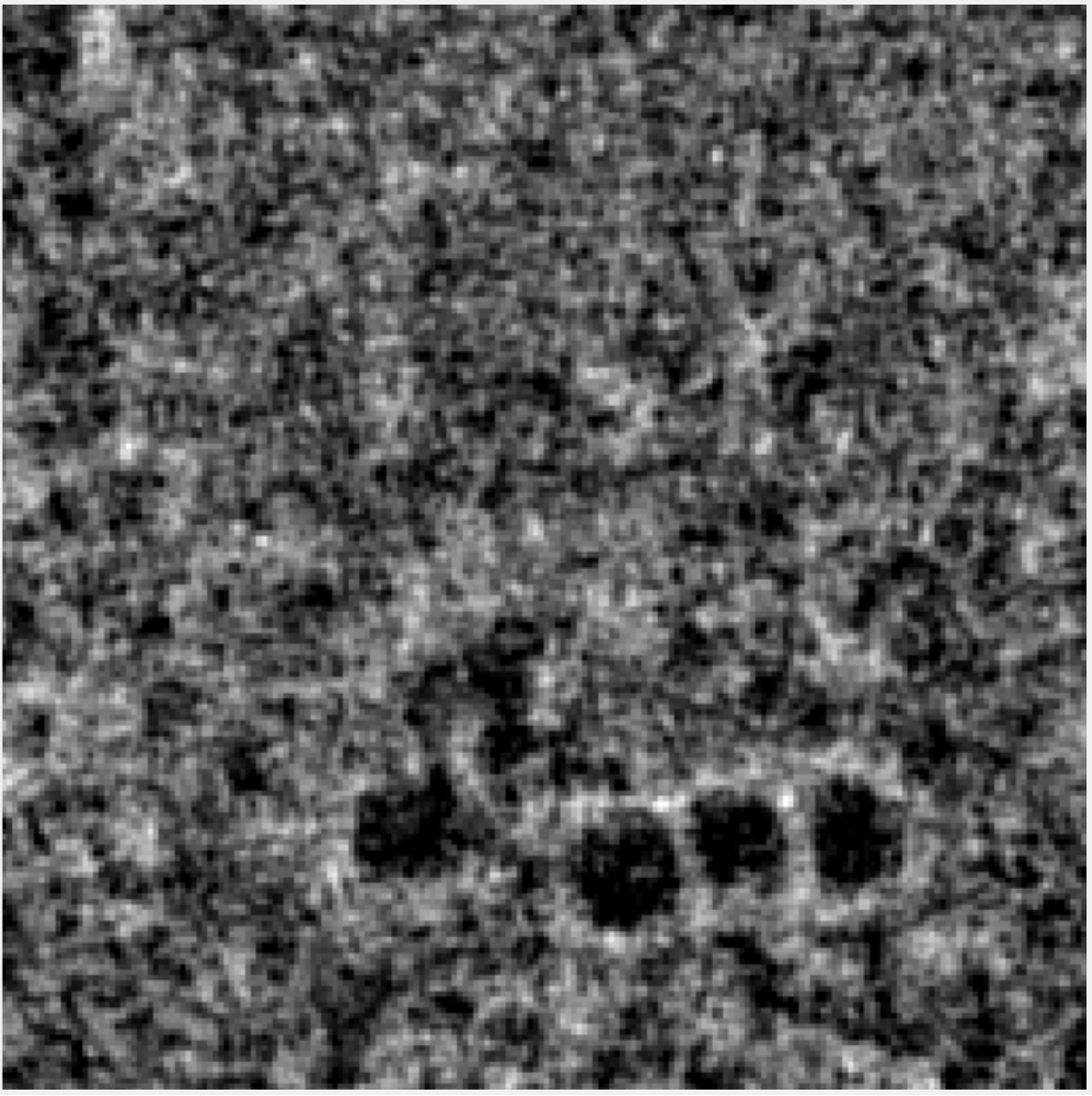}
	\vspace{1mm}\includegraphics[width=\ww]{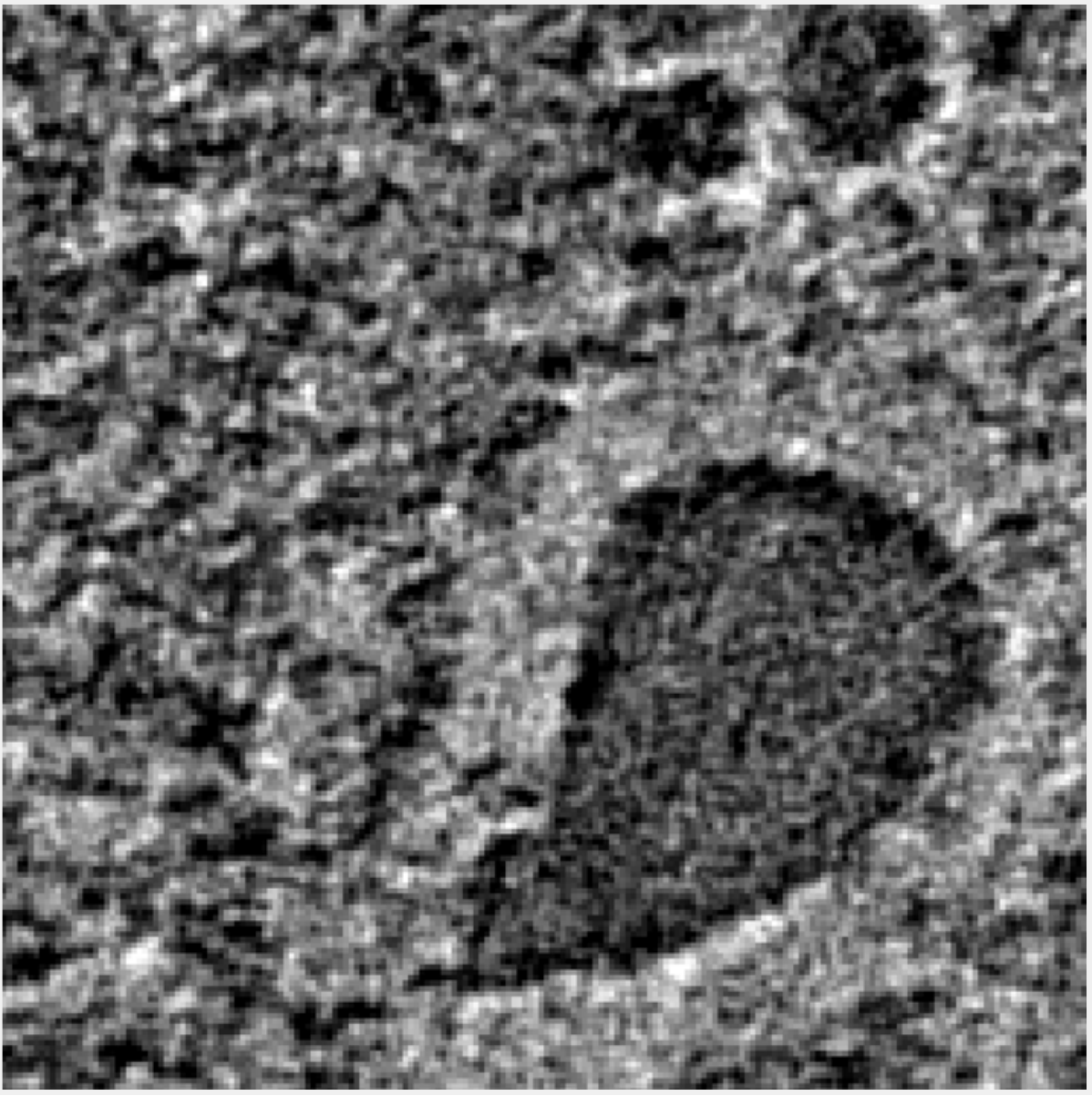}
	\end{center}
	\end{multicols}
		\caption{FBP reconstruction with artifact reduction (cf.
		Theorem \ref{thm:reduction}).}
		\label{subfig:sync data fbp rec ours}\end{subfigure}
	      \caption{Reconstructions from synchrotron data without smoothing
	      (top) and with smoothing (bottom) \cite[\copyright IOP
		Publishing. Reproduced by permission of IOP Publishing.  All rights
		reserved]{Borg2017}. }
	\label{fig:improved synchrotron reconstruction}
	\end{figure}

\section{Discussion}\label{sect:disc}

We first make observations about our results for $\LA$ and then
discuss generalizations.

\subsection{Observations} The proofs of Theorems \ref{thm:curve} and
\ref{thm:streak} show that if $\thpo\in \bd(A)$ and \hfil\newline
$\WF(\onea Rf) = T^*(\Sor)\smo$, then $\LA f$ will have a streak all
along $L\thpo$.  The analogous theorem for Sobolev singularities,
Theorem \ref{thm:Sobolev sing}\ref{nonsmooth bndy}, assumes that $A$
has a corner at $\thpo$.  If $A$ has a weaker singularity at $\thpo$,
then an analogous theorem would hold but one would need to factor in
the Sobolev strength of the wavefront of $\onea$ above $\thpo$.

The artifact reduction method, which is motivated by Theorem
\ref{thm:reduction}, works well for the synchrotron data as was shown
in Figure \ref{fig:improved synchrotron reconstruction} in Section
\ref{sect:synchrotron}. The article \cite{Borg2017} provides more
elaborate artifact reduction methods that are even more successful for
this particular problem.  We point out that this simple technique
might not work as efficiently in other incomplete data tomography
problems as in the problems we present.  Nevertheless, our theorems
and experiments show that abrupt cutoffs that add new singularities in
the sinogram should be avoided.

There are other methods to deal with incomplete data. For example,
data completion using the range conditions for the Radon transform has
been developed, e.g., in \cite{Louis:1980LA, BL,
VogelgesangSchorr2017}.  In \cite{PUW-streaks} and
\cite{CPWWS-susceptibility, PalamodovSusceptibility-2016}, the authors
develop artifact reduction methods for quantitative susceptibility
mapping.  For metal artifacts, there is vast literature (see, e.g.,
\cite{CT-Artifacts-BF:2012}) for artifact reduction methods, and we
believe that those methods might also be useful for certain other
incomplete data tomography problems. In \cite{ParkChoiSeo:2017,
Rigaud:2017, PUW-metal}, the authors have effectively used microlocal
analysis to understand these related problems.

Our theory is developed based on the continuous case -- we view the
data as functions on $\Sor$, not just defined at discrete points.  As
shown in this article, our theory predicts and explains the artifacts
and visible and invisible singularities.  In practice, real data are
discrete, and discretization may also introduce artifacts, such as
undersampling streaks.  Discretization in our synchrotron experiment
could be a factor in the streaks in Figure \ref{fig:sinogram-streaks}
in Section \ref{sect:synchrotron}.  Furthermore, numerical experiments
have finite resolution, and this can cause (and sometimes
de-emphasize) artifacts.  For all these reasons, further analysis is
needed to shed light on the interplay between the discrete and the
continuous theory for CT reconstructions from incomplete data.

\subsection{Generalizations}

Theorems \ref{thm:curve} and \ref{thm:streak} were proven for $\LA =
R^* \paren{\Lambda\paren{\onea R}}$, but the results hold for any
filtering operator that is elliptic in the sense of Remark
\ref{rem:Lambda}. This is true because that ellipticity condition is
all we used about $\Lambda$ in the proofs.  For example, the operator,
$L=-\frac{\partial^2}{\partial p^2}$, in Lambda CT \cite{FRS}
satisfies this condition, and the only difference comes in our Sobolev
Continuity Theorem \ref{thm:Sobolev sing}.  Since $L$ is order two,
the operator $R^* LR$ is of order $1$ and the smoothness in Sobolev
scale of the reconstructions would be one degree lower than for $\LA$.

Our theorems hold for fan-beam data when the source curve $\gamma$ is
smooth and convex and the object is compactly supported inside $\gamma$.
This is true because, in this case, the fan-beam parameterization of lines
is diffeomorphic to the parallel-beam parametrization we use and the
microlocal theorems we use are invariant under diffeomorphisms.  However,
one needs to check that the parallel-beam data set equivalent to the given
fan-beam data set satisfies Assumption \ref{hyp:A}.

Theorems \ref{thm:curve} and \ref{thm:streak} hold verbatim for
generalized Radon transforms with smooth measures on lines in $\rtwo$
because they all have the same canonical relation, given by
\eqref{def:C}, and the proofs would be done as for $\LA$ but using the
basic microlocal analysis in \cite{Quinto1980}.  

Analogous theorems hold for other Radon transforms including the
generalized hyperplane transform, the spherical transform of
photoacoustic CT, and other transforms satisfying the Bolker
assumption \eqref{Bolker}.  The proofs would use our arguments here
plus the proofs in \cite{FrikelQuinto-hyperplane,FrikelQuinto2015}.
These generalizations are the subject of ongoing work.  In incomplete
data problems for $R$, the artifacts are either on $x_{b}$-curves or
they are streaks on the lines corresponding to points on $\bd(A)$.
However, in higher-dimensional cases, the results will be more subtle
because artifacts can spread on \emph{proper subsets} of the surface
over which data are taken, not necessarily the entire set (see
\cite[Remark 4.7]{FrikelQuinto2015}).

Analogous theorems should hold for cone-beam CT, but this type of CT is
more subtle because the reconstruction operator itself can add artifacts,
even with complete data \cite{GU1989, FLU}.



\appendix\section{Proofs}

We now provide some basic microlocal analysis and then use this to prove
our theorems.  We adapt the standard terminology of microlocal analysis and
consider wavefront sets as subsets of cotangent spaces
\cite{Warner}. Elementary presentations of microlocal analysis for
tomography are in \cite{KrQu-chapter,Kuchment:RTbook}.  Standard references
include \cite{Friedlander98, Treves:1980vf}.

\subsection{Building blocks}\label{appendix:lemmas}

Our first lemma gives some basic facts about wavefront sets.

\begin{lemma}\label{lem:WF equal} Let $\xo\in \rtwo$.  Let $u$ and $v$
be locally integrable functions or distributions.
\begin{enumerate}[label=\Alph*.]

 \item\label{WF equal} Let $U$ be an open neighborhood of $\xo$.  Assume
   that  $u$ and $v$ are equal on $U$, then
   $\WF_{\xo}(u) = \WF_{\xo}(v)$.

\item\label{mult by smooth} If $u$ and $\psi$ are both in $\Lloc$ and
$\psi$ is smooth near $\xo$, then $\WF_{\xo}(\psi u)\subset
\WF_{\xo}(u)$.  If, in addition, $\psi$ is nonzero at $\xo$ then
$\WF_{\xo}(u) = \WF_{\xo}(\psi u)$.

\item \label{smooth at x} $\WF_{\xo}(u)= \emptyset$ if and only if
there is an open neighborhood $U$ of $\xo$ on which $u$ is a smooth
function.
\end{enumerate}

The analogous statements hold for functions on $\Sor$.
\end{lemma}

These basic properties are proven using the arguments in Section 8.1 of
\cite{Hoermander03}, in particular, Lemma 8.1.1, Definition 8.1.2, and
Proposition 8.1.3. This lemma is valid for functions on $\Sor$ using 
the identifications of $\Sor$ with $\rtwo$ given by \eqref{def:thbar} and 
for functions \eqref{def:tg}, and the fact that
singularities are defined locally.

Our next definition will be useful to describe how wavefront sets transform
under $R$ and $R^*$.

\begin{definition}\label{def:composition}  Let $C\subset
  T^*(\Sor)\times T^*(\rtwo)$ and let $B\subset T^*(\rtwo)$.  The
  composition is defined \[C\circ B = \sparen{(\th,p,\eta) \in T^*(\Sor)\st
  (\th,p,\eta, x,\xi)\in C \text{ for some } (x,\xi)\in B}.\] We define
  $C^t =
\sparen{(x,\xi,\th,p,\eta)\st (\th,p,\eta,x,\xi)\in C}$.  
\end{definition}

The function $g$ on $\Sor$ will be called \emph{symmetric} if \bel{g
  symmetry}\forall \thp\in \Sor, \ g(\th,p)=g(-\th,-p).\ee If $f\in \LD$,
then $Rf$ and $\Lambda \onea Rf$ are both locally integrable functions are
symmetric in this sense.  For such functions, \bel{WF(g)
  symmetry}\paren{\tho,\po,\omo(-\alpha\dth + \dpp)}\in \WF(g)
\Leftrightarrow \paren{-\tho,-\po,-\omo(\alpha\dth + \dpp)}\in \WF(g).\ee
For these reasons, we will identify cotangent vectors
\bel{identify}\paren{\tho,\po,\omo(-\alpha\dth + \dpp)}
\Leftrightarrow \paren{-\tho,-\po,-\omo(\alpha\dth + \dpp)}.\ee

Our next proposition is the main technical theorem of the
article.  It provides the wavefront correspondences for $R$ and $R^*$
which we will use in our proofs.  

\begin{proposition}[Microlocal correspondence of singularities]\label{prop:R*corresp}  
The X-ray transform, $R$, is an elliptic Fourier integral operator (FIO)
with canonical relation \bel{def:C}
\begin{aligned}C=&\Big\{\paren{\th,x\cdot\th,
    \om(-x\cdot\thperp\dth
    + \dpp), x,\om\th\dx}\\
  &\qquad\st \th\in \So,\, x\in \rtwo,\,\om\neq
0\Big\}.\end {aligned}\ee 

Let $f\in\LD$ and let $g$ be a locally integrable function on $\Sor$
that is symmetric by \eqref{g symmetry}.  Let $\xo\in \rtwo$, $\tho\in
\So$, and let $p$, $\alpha$, and $\omega$ be real numbers with
$\om\neq 0$.

The X-ray transform $R$ is an elliptic FIO with canonical relation
$C$.  Therefore,
\bel{WFcorresp}\begin{gathered}\WF(Rf)=C\circ
\WF(f)\qquad \text{and}\\
C\circ\sparen{(\xo,\omega\th\dx)} =
\sparen{\paren{\tho,\xo\cdot \tho,\omega(-\xo\cdot
\thoperp\dth+\dpp)}}\end{gathered}\ee under the identification
\eqref{identify}.

 The dual transform $R^*$ is an elliptic FIO with canonical relation $C^t$.
Then,
\bel{R*corresp}\begin{gathered}\WF(R^* g) = C^t\circ \WF(g)\quad
\text{and}\\
C^t\circ
\sparen{(\th,p,\om(-\alpha\dth +
\dpp))}=\sparen{(\xo(\th,p,\alpha),\om\th\dx)}\\
\text{where}\quad \xo(\th,p,\alpha) =
\alpha\thperp+p\th.\end{gathered}\ee 
\end{proposition}

Here are pointers to the elements of this proof. The facts about $R$ are
directly from \cite[Theorem 3.1]{Quinto93} or \cite[Theorem A.2]{Q:sc}, and
they use the calculus of the FIO $R$ \cite{Gu1985, GS1977} (see also
\cite{Quinto1980}).  Note that the crucial point is that $R$ is an elliptic
Fourier integral operator that satisfies the \emph{global Bolker
  assumption:} the natural projection \bel{Bolker}\Pi_L:C\to T^*(Y)\ \
\text{is an injective immersion,}\ee so \eqref{WFcorresp} holds for $R$.
A straightforward calculation using \eqref{def:C} shows that the
  global Bolker assumption holds. Note that we are using the
identification \eqref{identify} in asserting that \eqref{WFcorresp} is an
equality.  The proofs for $R^*$ are parallel to those for $R$ except they
involve the canonical relation for $R^*$, $C^t$, rather than $C$.

\begin{remark}\label{rem:odds and ends}
  In \cite{FrikelQuinto2015,FrikelQuinto-hyperplane} the authors prove
artifact characterizations for limited data problems for photoacoustic
CT and generalized hyperplane transforms.  One key is a fundamental
result on multiplying distributions, \cite[Theorem
8.2.10]{Hoermander03}.  If $u$ and $v$ are distributions on $\Sor$,
this theorem implies they can be multiplied \emph{as distributions} if
they satisfy the non-cancellation condition $ \forall\,(\th,p,\eta)\in
\WF(u),\; (\th,p,-\eta)\notin \WF(v)$.  Then $uv$ is a distribution
and an upper bound for $\WF(uv)$ is given in terms of $\WF(u)$ and
$\WF(v)$.

  However, this non-cancellation condition does not hold for $\onea$ and
  $Rf$ when $\onea$ either is smooth with small slope or is not smooth at
  $\thpo$.  That is why we consider functions $f\in \LD$ in this article
  since $\onea Rf$ will be a function in $L^2(\Sor)$ even if \cite[Theorem
  8.2.10]{Hoermander03} does not apply.
\end{remark}

Our next remark will be used in ellipticity proofs that follow.

\begin{remark}\label{rem:Lambda} The operator $\Lambda$ is elliptic in
all cotangent directions except $\dth$ because the symbol of $\Lambda$
is $\abs{\tau}$ where $\tau$ is the Fourier variable dual to $p$.
However, the $\dth$ direction will not affect our proofs.  This is
true because, for any function $f\in \LD$, the covector
$(\th,p,\om\dth)$ is not in $\WF(Rf)$ because $\WF(Rf)= C\circ \WF(f)$
(use the definition of composition and \eqref{def:C}).  So, for each
$f\in \LD$, $\WF(\Lambda Rf)=\WF(Rf)$.  Because
$C^t\circ\sparen{(\th,p,\alpha\dth)}=\emptyset$ by \eqref{def:C}, even if
$(\th,\,\om\dth)\in \WF(\onea Rf)$, that covector will not affect the
calculation of $C^t\circ \WF(\Lambda \onea Rf)$.  Therefore, $\Lambda$
is elliptic on all cotangent directions that are preserved when
composed with $C^t$, and these are all the directions we need in our
proofs.  

Our theorems will be valid for any pseudodifferential operator on $\Sor$
that is invariant under the symmetry condition \eqref{g symmetry} and
satisfies this ellipticity condition (although the Sobolev results
will depend on the order of the operator). 
\end{remark}

\subsection{Proof of Theorems \ref{thm:curve}, \ref{thm:streak}, and
\ref{thm:complete}}\label{appendix:added}

In the proofs of these theorems, we use Proposition
\ref{prop:R*corresp} to analyze how multiplication by $\onea$ adds
singularities to the data $Rf$ and then to the reconstruction, $\LA
f$. We first make observations that will be useful in the proofs. 

 Let $A$ satisfy Assumption \ref{hyp:A} and let $f\in \LD$. Let 
 \[G=\onea R f  \quad\text{then}\quad R^*\Lambda G = \LA f.\]
By Remark \ref{rem:Lambda} and the statements in Proposition
\ref{prop:R*corresp}, \bel{LAG}\WF(\LA f)=C^t\circ \WF(G).\ee
Using the expression \eqref{def:C} for $C$, one can show for $\thpo\in
\Sor$ that \bel{fullC}
\begin{gathered}
C\circ\paren{N^*(L\thpo)\smo}= T^*_\thpo(\Sor) \setminus P \\
\text{where}\qquad N^*(L\thpo) =\sparen{(x,\omega\tho\dx)\st
x\in L\thpo,\ \omega\in \rr}\\
\text{and}\qquad P=\sparen{(\th,p, \omega\dth)\st \thp\in \Sor,\
\omega\in \rr}.
\end{gathered}\ee
Because $\WF(Rf)=C\circ \WF(f)$, \eqref{fullC} implies that
if $f$ is smooth conormal to $L\thpo$, then $Rf$ is smooth near $\thpo$.

 Using analogous arguments for $C^t$, one shows for $\thp\in \Sor$
that \bel{fullCt}C^t\circ\paren{T^*_\thpo(\Sor)\smo} = N^*(L\thpo)\smo
.\ee By \eqref{LAG}, if $G$ is smooth near $\thpo$ then $\LA f$
is smooth conormal to $L\thpo$.
\bigskip

To start the proofs, let $f\in \LD$ and let $A$ be a data set
satisfying Assumption \ref{hyp:A}.  Theorem \ref{thm:visible}
establishes that if $\thpo\notin \bd(A)$, then there are no artifacts
in $\LA f$ conormal to $L\thpo$ (since $\WFLo(\LA f)\subset
\WFLo(f)$).  Therefore, the only singular artifacts are on lines
$L\thpo$ for $\thpo\in \bd(A)$.  

\begin{proof}[Proof of Theorem \ref{thm:curve}]
Assume
$\bd(A)$ is smooth with finite slope at $\thpo$.  Therefore, there is
an open neighborhood $I$ of $\tho$ and a smooth function $p=p(\th)$
for $\th\in I$ such that $(\th,p(\th))\in \bd(A)$.  A straightforward
calculation shows for each $\th\in I$ and each $\omega\neq 0$ that 
\[\eta(\th) = \paren{\th,p(\th),\omega\paren{-p'(\th)\dth+\dpp}}\]
is conormal to $\bd(A)$ at $(\th,p(\th))$.  A calculation using
\eqref{R*corresp} and \eqref{LAG} shows that \bel{eta-LA}\eta(\th)\in
\WF(G) \quad\text{if and only if}\quad
(\xb(\th),\omega\th\dx)\in WF(\LA f),\ee
where $\xb(\th)$ is given by \eqref{def:xb}.
Then, $(\xb(\tho),\omega\tho\dx)$ is the possible
object-independent artifact that could occur on $L\thpo$. Note that
$\xb(\th)$ is simply the $x$-projection of $C^t\circ N^*(\bd(A))$.

By taking the derivative $\xb'(\th)$, one can show that the only case
in which the $\xb$-curve is a subset of a line occurs when $\bd(A)$ is
locally defined by lines through a point (e.g., for some $\xo\in
\rtwo$, $\bd(A)$ is locally given by $p(\th)=\xo\cdot\th$).  However,
in this case \eqref{def:xb} shows that the $\xb$-curve is the single
point $\xo$. This proves part \ref{not a line}

If $f$ has no singularities conormal to $L\thpo$, then $Rf$ is smooth
near $\thpo$, so
\newline
 $\WFpto(G)\subset \WFpto(\onea)$ by Lemma \ref{lem:WF
equal}\,\ref{mult by smooth}\ This proves part \ref{no sing}\ref{no sing:WF point}.

If $Rf$ is zero in a neighborhood of $\thpo$, then $G$ is smooth near
$\thpo$ so, by the note below \eqref{fullCt}, $\LA f$ is smooth conormal to
$L\thpo$.  This proves part \ref{no sing}\ref{Rf=0}.

If $Rf\thpo\neq 0$, then $\WF_\thpo(G)=\sparen{\eta(\tho)}$ by Lemma
\ref{lem:WF equal}\,\ref{mult by smooth}\  Now, by
\eqref{eta-LA},\hfil\newline
$(\xb(\tho),\omega\tho\dx)\in \WF(\LA f)$.  This proves part \ref{no sing}\ref{Rf neq 0}
and finishes the proof of part  \ref{no sing}
\end{proof}

 \begin{proof}[Proof of Theorem \ref{thm:streak}]

To prove part \ref{thmPart:object dep} we make a simple observation.
Singularities of $f$ conormal to $L\thpo$ can cause singularities in
$G$ only above $\thpo$ and those can cause singularities of $\LA f$
only conormal to $L\thpo$.  

Part \ref{thmPart:smooth object depII} follows from the fact that the
conormal to $\bd(A)$ at $\tho$ is $\om \dth$ for $\om\neq 0$, that
$C^t\circ\sparen{(\th,p,\om\dth}=\emptyset$, and the arguments in the
proof of Theorem \ref{thm:curve}\ref{no sing}\ref{no sing:WF point}.

Now, we assume $\bd(A)$ is not smooth at $\thpo$. 

The first observation is straightforward: if $\bd(A)$ is not smooth at
$\thpo$, then that singularity can cause singularities in $G$ at $\thpo$
which cause singularities of $\LA f$ conormal to $L\thpo$ (and nowhere
else).

Assume $f$ is smooth conormal to $L\thpo$, $Rf\thpo\neq 0$, and $A$
has a corner at $\thpo$ (see Definition \ref{def:smooth}). Then, by
Lemma \ref{lem:WF equal}, $\WF_\thpo (G) = \WF_\thpo(\onea)$ which is
equal to $T^*_\thpo(\Sor)\smo$.  Therefore, by \eqref{fullCt},
$\WF_{L\thpo}(\LA f)=N^*(L\thpo)\smo$.  This finishes the proof of
Theorem \ref{thm:streak}
\end{proof}

\begin{proof}[Proof of Theorem \ref{thm:complete}] Let $f\in \LD$ and
assume $A$ satisfies Assumption \ref{hyp:A}.  Theorem
\ref{thm:visible} establishes that artifacts are added in $\LA f$
conormal to $L\thpo$ only when $\thpo\in\bd(A)$.  Let $\thpo\in \bd(A)$.
Singularities of $G=\onea Rf$ at $\thpo$ come only from singularities of
$\onea$ or singularities of $Rf$ at $\thpo$.  Therefore, singularities of
$\LA f$ conormal to $L\thpo$ come only from singularities of $\onea$ at
$\thpo$ or singularities of $Rf$ at $\thpo$.

The artifacts of $\LA f$ caused by $\onea$ are analyzed in the proof
of Theorem \ref{thm:curve} and Theorem \ref{thm:streak} parts
\ref{thmPart:smooth object depII} and \ref{thmPart:nonsmooth} The
artifacts of $\LA f$ caused by $Rf$ are covered in Theorem
\ref{thm:streak}\,\ref{thmPart:object dep} This takes care of all
 singular artifacts for the continuous problem.\end{proof}

 \subsection{Proof of Theorem \ref{thm:Sobolev
sing}}\label{appendix:Sobolev Sing} 

We first prove a proposition giving the correspondence between Sobolev
wavefront set and $R^*$.

\begin{proposition}[Sobolev wavefront correspondence for $R$ and 
$R^*$]\label{prop:Sobolev R*} Let $\thpo\in \Sor$, $\omo\neq 0$, and
let $s$ and $\alpha$ be real numbers.  Let
\[\etao =\omo(-\alpha \dth+\dpp), \quad \xo= \po\tho+\alpha
\thoperp, \text { and }\xio=\omo\tho\dx.\] Let $f$ be a
distribution on $\rtwo$ and $g$ a distribution on $\Sor$. Then,
\begin{align}
\label{app:SobolevCorrespR}
 (\xo,\xio)\in \WF_s(f) &\Longleftrightarrow (\tho,\po,\etao)\in
WF_{s+1/2}(Rf),\\
\label{app:SobolevCorrespR*}
(\tho,\po,\etao)\in \WF_s(g)& \Longleftrightarrow (\xo,\xio)\in
WF_{s+1/2}(R^*g). 
\end{align}
\end{proposition}

\begin{proof}
Equivalence \eqref{app:SobolevCorrespR} is given \cite[Theorem
3.1]{Quinto93}, however the proof of the $\Leftarrow$
implication for $R$ was left to the reader.

The proof of the $\Rightarrow$ implication of \eqref{app:SobolevCorrespR*}
is completely analogous to the proof given in \cite{Quinto93} for $R$.  For
completeness, we will prove the $\Leftarrow$ implication of
\eqref{app:SobolevCorrespR*}.  Assume $g$ is in $H_{s}$ at
$(\tho,\po,\etao)$. By \cite[Theorem 6.1, p.\ 259]{Petersen}, we can write
$g=g_1+g_2$ where $g_1\in H_s$ and $(\tho,\po,\etao)\notin \WF(g_2)$.  The
operator $R^*$ is continuous in Sobolev spaces from $H_s$ to
$H_{s+1/2}^\text{loc}$ by \cite[Theorem VIII 6.1]{Treves:1980vf} since
$C^t$ is a local canonical graph.  Therefore
$R^*g_1\in H^\text{loc}_{s+1/2}$.  Since $(\tho,\po,\etao)\notin \WF(g_2)$,
$(\xo,\xio)\notin \WF(R^*g_2)$ by the wavefront correspondence
\eqref{R*corresp}.  An exercise using Definition \ref{def:Hs norm} and the
Fourier transform shows that $R^*g = R^*g_1+R^*g_2$ is in $H_{s+1/2}$ at
$(\xo,\xio)$.
\end{proof}

 \begin{proof}[Proof of Theorem \ref{thm:Sobolev sing}] Let $f\in \LD$
and let $A$ satisfy Assumption \ref{hyp:A}.  Let $\thpo\in \bd(A)$ and
assume $Rf\thpo\neq 0$ and $f$ is smooth conormal to $L\thpo$.
Because $f$ is smooth conormal to $L\thpo$, $\WF_\thpo(Rf) =
\emptyset$ so $Rf$ is smooth in a neighborhood of $\thpo$ by Lemma
\ref{lem:WF equal}\,\ref{smooth at x} Since $Rf\thpo\neq 0$, for each
$s$, \bel{Sobolev=}\paren{\WF_{s-1}}_\thpo(\Lambda\onea
Rf)=\paren{\WF_s}_\thpo(\onea Rf) = \paren{\WF_s}_\thpo(\onea)\,;\ee
the left-hand equality is true because $\Lambda$ is an elliptic
pseudodifferential operator of order one (except in the irrelevant
direction $\dth$---see Remark \ref{rem:odds and ends}), and the
right-hand equality is true by Lemma \ref{lem:WF equal}\,\ref{mult by
smooth}

To prove part \ref{smooth bndy not vertical} of the theorem, assume
$\bd(A)$ is smooth and has finite slope at $\thpo$.  Because the
Sobolev wavefront set is contravariant under diffeomorphism
\cite{Treves:1980vf}, we may assume $\bd(A)$ is a horizontal line, at
least locally near $\thpo$.  Let $\etao = \dpp$.  We claim that
$(\tho,\po, \pm\etao)\in \WF_{1/2}(\onea)$ and, for $s<1/2$, $\onea $
is in $H_s$ at $ (\tho, \po,\pm \etao)$.  Furthermore $\onea$ is
smooth in every other direction above $\thpo$.  The proofs of these
two statements are now outlined.  Using a product cutoff function
$\psi=\psi_1(\th)\psi_2(p)$ to calculate $\Fc(\psi \onea)$ and
integrations by parts, one can show that this localized Fourier
transform is of the form $S(\nu)T(\tau)$ where $S$ is a smooth,
rapidly decreasing function and $T$ is $\Oc(1/\abs{\tau})$ (and not
$\Oc(1/\abs{\tau}^p$ for any $p>1$).  Therefore $S(\nu)T(\tau)$ is
rapidly decaying in all directions but the vertical.  This implies
that $\onea$ is in $H_s$ for $s<1/2$ at $(\tho,\po,\pm\etao)$ and
$(\tho,\po,\pm\etao)\in\WF_{1/2}(\onea)$.  This also shows that this
localized Fourier transform is rapidly decaying in all directions
except $\pm\etao$.  Now, using \eqref{Sobolev=} one sees that
$(\tho,\po,\pm\etao)\in \WF_{-1/2}(\Lambda\onea Rf)$; $\Lambda \onea
Rf$ is in $H_s$ for $s<-1/2$ at $(\tho,\po,\pm\etao)$; and $(\tho,\po,
\eta)\notin \WF(\Lambda\onea Rf)$ for any $\eta$ not parallel to
$\etao$.

Now, by Proposition \ref{prop:Sobolev R*}, $\LA f =R^*\Lambda \onea Rf$ is
in $H_s$ at $(\xb(\tho),\pm\tho\dx)$ for $s<0$ and \[(\xb(\tho),
\pm\tho\dx)\in WF_{0}(\LA f), \] where $\xb(\tho)$ is given by
\eqref{def:xb}.  Using this theorem again, one sees that for any $x\in
L\thpo$, if $x\neq \xb(\tho)$,
\[(x,\pm \tho\dx)\notin \WF(\LA f).\]  Therefore,  
the only covectors in $N^*(L\thpo)\cap \WF(\LA f)$ are
$(\xb(\tho),\alpha \tho\dx)$ for $\alpha \neq 0$.

To prove part \ref{nonsmooth bndy}, assume $\bd(A)$ has a corner at
$\thpo$.  Let $\alpha_1$ and $\alpha_2$ be the slopes at $\thpo$ of the two
parts of $\bd(A)$.  Let \bel{two xbs} \eta_j = -\alpha_j \dth + \dpp, \quad
{\xb}_j=\po\tho + \alpha_j\thoperp, \qquad j=1,2. \ee An argument similar
to the diffeomorphism/integration by parts argument in the last part of the
proof is used.  First a diffeomorphism is used to transform the corner so,
\emph{locally} $A$ becomes $\lA=\sparen{(\th,p)\st \th\geq 0, p\geq 0}$.
To do this, one uses Definition \ref{def:smooth} and footnote
  \ref{footnote:corner} and the Inverse and Implicit Function Theorems.
  Then one uses a product cutoff $\psi=\psi_1(\th)\psi_2(p)$ to calculate
  $\WF_s(\one_{\lA})$ at $(0,0)$.  Then, the Fourier transform can be
  written $\Fc\paren{\psi \one_{\lA}}=S(\nu)T(\tau)$ where
  $S(\nu)=\Oc(1/\abs{\nu})$ and $T(\tau)=\Oc(1/\abs{\tau})$.  So, the
  localized Fourier transform is decreasing of order $-1$ in the $\dpp$
  (vertical) and $\dth$ (horizontal) directions and $-2$ in all other
  directions.  

Note that $\eta_1$ and $\eta_2$ are the images of $\dpp$ and $\dth$
under the diffeomorphism back to the original coordinates.  By
contravariance of Sobolev wavefront set under diffeomorphism and the
assumption that $Rf$ is smooth and nonzero near $\thpo$,
$(\tho,\po,\pm\eta_j)\in \WF_{-1/2}(\Lambda\onea Rf)$ and, for
$s<-1/2$, $\Lambda \onea Rf$ is in $H_s$ at $(\tho,\po,\eta_j)$.
Other covectors are in $\WF_{1/2}(\Lambda\onea Rf)$.  One finishes the
proof using \eqref{app:SobolevCorrespR*}.

This proof shows for $j=1,2$ that
$C^t\circ \sparen{(\tho,\po,\eta_j)}\in \WF_0(\LA f)$, and these are the
``more singular points'' referred to after the statement of Theorem
\ref{thm:Sobolev sing}.  If one part of $\bd(A)$ is vertical at $\thpo$,
then for one value of $j$, $\eta_j$ is parallel to $\dth$ and
$C^t\circ\sparen{(\tho,\po,\eta_j)}=\emptyset$ so there is only one point,
not two, on $L\thpo$ on which $f$ is more singular.
\end{proof}

\section*{Acknowledgements}

We thank the Japan Synchrotron Radiation Research Institute for the
allotment of beam time on beamline BL20XU of SPring-8 (Proposal 2015A1147)
that provided the raw data described in Section \ref{sect:synchrotron}.
Todd Quinto thanks John Schotland and Guillaume Bal for a stimulating
discussion about multiplying distributions that relates to Remark
\ref{rem:odds and ends} and the reasons we consider only functions in this
article.  He thanks Plamen Stefanov for stimulating discussions about these
results.  He thanks the Technical University of Denmark and DTU Compute for
a wonderful semester during which this research was being done and he is
indebted to his colleagues there for stimulating, enjoyable discussions
that influenced this work.  In addition, the authors thank the funding
agencies listed at the start of the article.  Quinto also thanks the Tufts
University Deans of Arts and Sciences for their support for his semester at
DTU.

The authors thank the four referees of this article for important comments
that greatly improved the exposition and theorems, and in particular, the
statements and proofs of Theorems \ref{thm:curve} and \ref{thm:streak}.
The original version of this article was posted to the arXiv on July 4,
2017.




\end{document}